\documentclass[A4,10pt]{iopart}
\usepackage{amsfonts}
\usepackage{geometry}
\usepackage{graphicx}

\newtheorem{theorem}{Theorem}
\newtheorem{axiom}{Axiom}
\newtheorem{definition}{Definition}
\newtheorem{corollary}{Corollary}
\newtheorem{remark}{Remark}
\newenvironment{proof}[1][Proof]{\noindent\textbf{#1.} }{\ \rule{0.5em}{0.5em}}

\begin{document}

\title{Fluctuation geometry: A counterpart approach of inference geometry}

\author{L. Velazquez}
\address{Departamento de F\'\i sica, Universidad Cat\'olica del Norte, Av. Angamos 0610, Antofagasta,
Chile.}

\begin{abstract}
Starting from an axiomatic perspective, \emph{fluctuation geometry} is developed as a counterpart approach of inference geometry. This approach is inspired on the existence of a notable analogy between the general theorems of \emph{inference theory} and the the \emph{general fluctuation theorems} associated with a parametric family of distribution functions $dp(I|\theta)=\rho(I|\theta)dI$, which describes the behavior of a set of \emph{continuous stochastic variables} driven by a set of control parameters $\theta$. In this approach, statistical properties are rephrased as purely geometric notions derived from the \emph{Riemannian structure} on the manifold $\mathcal{M}_{\theta}$ of stochastic variables $I$. Consequently, this theory arises as an alternative framework for applying the powerful methods of differential geometry for the statistical analysis. Fluctuation geometry has direct implications on statistics and physics. This geometric approach inspires a Riemannian reformulation of Einstein fluctuation theory as well as a geometric redefinition of the information entropy for a continuous distribution.
\end{abstract}

\section{Introduction}

Inference theory supports the introduction of a Riemannian distance notion \cite{Amari}:
\begin{equation}\label{inf.dist}
ds^{2}_{F}=g_{\alpha\beta}(\theta)d\theta^{\alpha}d\theta^{\beta}
\end{equation}
to characterize the \textit{statistical distance} between two close members of a generic parametric family of distribution functions:
\begin{equation}\label{DF}
dp(I|\theta)=\rho(I|\theta)dI.
\end{equation}
Here, the metric tensor $g_{\alpha\beta}(\theta)$ is provided by the so-called \emph{Fisher's information matrix} \cite{Fisher}. The existence of this type of Riemannian formulation was pioneering suggested by Rao \cite{Rao}, which is hereafter referred to as \textit{inference geometry}\footnote{This approach is referred to as \emph{Riemannian geometry on statistical manifolds} in the literature.  However, the denomination \emph{inference geometry} is employed here to avoid the ambiguity with \emph{fluctuation geometry}, which is also a Riemannian geometry on statistical manifold.}. Inference geometry provides very strong tools for proving results about statistical models, simply by considering them as well-defined geometrical objects. As expected, inference theory and its geometry have a direct application in those physical theories with a statistical apparatus as statistical mechanics and quantum theories. A curious example is the so-called \emph{extreme physical information principle}, proposed by Frieden in 1998, which claims that all physical laws can be derived from purely inference arguments \cite{Frieden}.
Inference geometry has been adapted to the mathematical apparatus of quantum mechanics \cite{Brody1,Barndorff-Nielsen,Gibilisco,Gibilisco2}. In fact, modern interpretations of uncertainty relations are inspired on their arguments \cite{Holevo,Caianiello,Braunstein}. Inference geometry has been successfully employed in statistical mechanics to study phase transitions \cite{Janke,Brody}, as well as in the framework of so-called thermodynamics geometry \cite{Crooks}.

The goal of this paper is to show that the parametric family of distribution functions (\ref{DF}) supports the introduction of an alternative Riemannian distance notion:
\begin{equation}\label{fluct.dist}
ds^{2}=g_{ij}(I|\theta)dI^{i}dI^{j},
\end{equation}
which characterizes the statistical distance between two close sets of continuum stochastic variables $I$ and $I+dI$ for fixed values of control parameters $\theta$. This geometrical approach, hereafter referred to as \textit{fluctuation geometry}, is inspired on the existence of a notable analogy between the general theorems of inference theory and some general fluctuation theorems recently proposed in the framework of equilibrium classical fluctuation theory \cite{Vel.ETFR,Vel.URSM,Vel.GEFT,Vel.GEO}. The main consequence derived from this analysis is the possibility to rephrase the original parametric family of distribution functions (\ref{DF}) in terms of purely geometry notions. This connection enables the direct application of powerful tools of differential geometry for the analysis of their \emph{absolute statistical properties}\footnote{Tensorial formalism of Riemannian geometry allows to study the \emph{absolute geometric properties} of a manifold $\mathcal{M}$ using any of its coordinate representations $\mathcal{R}_{I}$. A relevant example of absolute property is the \emph{curvature} of the manifold $\mathcal{M}$, which is manifested in any coordinate representation $\mathcal{R}_{I}$. In general relativity theory, the curvature of space-time $\mathbb{M}^{4}$ is identified with gravitation interaction. The effects of gravitation are \emph{absolute} or \emph{irreducible}, while the effects associated with the \emph{inertial forces} are \emph{reducible}. In fact, the existence of inertial forces crucially depends on the reference frame, that is, the specific coordinate representation of the space-time $\mathbb{M}^{4}$.}.

The paper is organized as follows. For the sake of self-consistence, the next section will be devoted to discuss the main motivation of the present proposal: the analogy between inference theory and fluctuation theory. Afterwards, it will be developed an axiomatic formulation of fluctuation geometry. Firstly, the postulates of fluctuation geometry are presented in section \ref{Postulates}. Then, their consequences are considered in section \ref{Consequences} to perform a geometric reinterpretation of the statistical description. Section \ref{Examples} is devoted to discuss two simple application examples of fluctuation geometry. Implications of the present approach in some statistics and physics problems will be analyzed in section \ref{Relevance}, as example, a reconsideration of information entropy for a continuous distribution and the development of a Riemannian reformulation of Einstein fluctuation theory. Final remarks and open problems will be summarized in section \ref{Final}.

\section{Motivation}
Let us start from the parametric family of distribution functions (\ref{DF}), which describes the behavior of a set of continuous stochastic variables $I$ driven by a set $\theta$ of control parameters. Let us denote by $\mathcal{M}_{\theta}$ the statistical manifold constituted by all admissible values of the stochastic variables $I$ that are accessible for a given value $\theta$ of control parameters, which is hereafter assumed as a simply connected domain. Moreover, let us denote by $\mathcal{P}$ the statistical manifold constituted by all admissible values of control parameters $\theta$ (each point $\theta\in\mathcal{P}$ represents a given distribution function). The parametric family of distribution functions (\ref{DF}) can be analyzed from two different perspectives:
\begin{itemize}
\item To study the \emph{fluctuating behavior} of stochastic variables $I\in\mathcal{M}_{\theta}$, which is the main interest of \emph{\textbf{%
fluctuation theory}} \cite{Vel.GEFT};

\item To study the relationship between this fluctuation behavior and the \emph{external control} described by the parameters $\theta\in\mathcal{P}$, which is the interest of \emph{\textbf{inference theory}} \cite{Fisher}.
\end{itemize}

\subsection{Fluctuation theory}
Let us admit that the probability density $\rho(I|\theta)$ is everywhere finite and differentiable, and obeys the following conditions for every point $I_{b}$ located on the boundary $\partial\mathcal{M}_{\theta}$ of the statistical manifold $\mathcal{M}_{\theta}$, $I_{b}\in\partial\mathcal{M}_{\theta}$:
\begin{equation}\label{boundary}
\lim_{I\rightarrow I_{b}}\rho(I|\theta)=\lim_{I\rightarrow
I_{b}}\frac{\partial}{\partial I^{i}}\rho(I|\theta)=0.
\end{equation}
The probability density $\rho(I|\theta)$ can be considered to introduce the \textit{differential forces} $\eta_{i}(I|\theta)$:
\begin{equation}\label{DGF}
\eta_{i}(I|\theta)=-\frac{\partial}{\partial
I^{i}}\log\rho(I|\theta).
\end{equation}
By definition, the differential forces $\eta_{i}(I|\theta)$ vanish in those points where the probability density $\rho(I|\theta)$ exhibits its local maxima or its local minima. The global (local) maximum of the probability density can be regarded as a \textit{stable (metastable) equilibrium points} $\bar{I}$, which can be obtained from the following \emph{stationary and stability conditions}:
\begin{equation}\label{stat.cond}
-\frac{\partial}{\partial
I^{i}}\log\rho(\bar{I}|\theta)=0,\:-\frac{\partial^{2}}{\partial
I^{i}\partial
I^{j}}\log\rho(\bar{I}|\theta)\succ0,
\end{equation}
where $A_{ij}\succ 0$ denotes that the matrix $A_{ij}$ is positive definite. In general, the differential forces $\eta_{i}(I|\theta)$ characterize the deviation of a given point $I\in\mathcal{M}_{\theta}$ from these local equilibrium points. Analogously, it is convenient to introduce the \emph{response matrix} $\chi_{ij}(I|\theta)$:
\begin{equation}\label{response}
\chi_{ij}(I|\theta)=\partial_{j}\eta_{i}(I|\theta),
\end{equation}
where $\partial_{i}A=\partial A/\partial I^{i}$, which describes the response of differential forces $\eta_{i}(I|\theta)$ under an infinitesimal change of the variable $I^{j}$.

As stochastic variables, the expectation values of the differential forces $\eta_{i}=\eta_{i}(I|\theta)$ identically vanish:
\begin{equation}\label{equi.cond}
\left\langle \eta_{i}\right\rangle =0,
\end{equation}
and these quantities also obey the \textit{fundamental and the associated fluctuation theorems} \cite{Vel.GEFT}:
\begin{eqnarray}\label{fund.fluct}
\left\langle \eta_{i}\delta I^{j}\right\rangle
=\delta_{i}^{j},\\
\left\langle
\chi_{ij}\right\rangle=\left\langle\eta_{i}\eta_{j}\right\rangle,
\label{assoc.fluct}
\end{eqnarray}
where $\delta^{j}_{i}$ is the Kronecker delta. The previous theorems are derived from the following identity:
\begin{equation}\label{identity}
\left\langle\partial_{i}
A(I|\theta)\right\rangle=\left\langle\eta_{i}(I|\theta)A(I|\theta)\right\rangle
\end{equation}
substituting the cases $A(I|\theta)=1$, $I^{i}$ and $\eta_{i}$, respectively. Here, $A(I)$ is a differentiable function defined on the continuous variables $I$ with definite expectation values $\left\langle\partial A(I|\theta)/\partial I^{i}\right\rangle$ that obeys the following the boundary condition:
\begin{equation}
\lim_{I\rightarrow I_{b}}A(I)\rho(I|\theta)=0.
\end{equation}
Moreover, equation (\ref{identity}) follows from the integral expression:
\begin{equation}
\int_{\mathcal{M}_{\theta}}\frac{\partial\upsilon^{j}(I|\theta)}{\partial
I^{j}}\rho(I|\theta)dI  =\oint_{\partial\mathcal{M}_{\theta}}\rho(I|\theta)%
\upsilon^{j}(I|\theta)\cdot
d\Sigma_{j}  -\int_{\mathcal{M}_{\theta}}\upsilon^{j}(I|\theta)\frac{\partial\rho(I|
\theta)}{\partial I^{j}}dI
\end{equation}
derived from the \textit{intrinsic exterior calculus} of the statistical manifold $\mathcal{M}_{\theta}$ and the imposition of the constraint $\upsilon^{j}(I|\theta)\equiv\delta^{j}_{i}A(I|\theta)$. It is easy to realize that the identity (\ref{equi.cond}) and the associated fluctuation theorem (\ref{assoc.fluct}) are just the \emph{stationary and stability equilibrium conditions} (\ref{stat.cond}) \emph{written in term of statistical expectation values}, respectively. In particular, the positive definite character of the self-correlation matrix $M_{ij}(\theta)=\left\langle\eta_{i}(I|\theta)\eta_{j}(I|\theta)\right\rangle$ implies the positive definition of the matrix:
\begin{equation}\label{response}
\left\langle\partial_{i}\eta_{j}(I|\theta)\right\rangle=\left\langle-\frac{\partial^{2}\log\rho(I|\theta)}{\partial
I^{i}\partial I^{j}}\right\rangle\succ0,
\end{equation}
Remarkably, the fundamental fluctuation theorem (\ref{fund.fluct}) suggests the \emph{statistical complementarity} of the variable $I^{i}$ and its conjugated differential force $\eta_{i}=\eta_{i}(I|\theta)$. Using the Schwartz inequality $\left\langle \delta A \delta B\right\rangle^{2}\leq \left\langle \delta A^{2}\right\rangle \left\langle\delta B^{2}\right\rangle$, one obtains the following inequality:
\begin{equation}\label{unc.fluct}
\Delta I^{i} \Delta \eta_{i}\geq 1,
\end{equation}
where $\Delta x =\sqrt{\left\langle \delta x^{2}\right\rangle}$ is the \emph{statistical uncertainty} of the quantity $x$. Clearly, this last inequality exhibits the same mathematical appearance of \emph{Heisenberg's uncertainty relation} $\Delta q \Delta p \geq \hbar$. Recently, this result was employed to show the existence of uncertainty relations involving conjugated thermodynamic quantities \cite{Vel.ETFR,Vel.URSM}.  Equation (\ref{unc.fluct}) can be generalized considering the inverse $M^{ij}(\theta)$ of the self-correlation matrix of the differential forces $M_{ij}(\theta)=\left\langle\eta_{i}(I|\theta)\eta_{j}(I|\theta)\right\rangle$.  Denoting by $C^{ij}(\theta)=\left\langle\delta I^{i}\delta
I^{j}\right\rangle$ the self-correlation matrix of the stochastic variables $I$, it is possible to obtain the following matrical inequalities:
\begin{equation}
C^{ij}(\theta)-M^{ij}(\theta)\succeq 0.
\end{equation}
This last inequality is directly obtained from the positive definition of the self-correlation matrix $K^{ij}(\theta)=\left\langle J^{i}(I|\theta)J^{j}(I|\theta)\right\rangle$ of the auxiliary quantities $J^{i}(I|\theta)=\delta I^{i}+M^{ij}(\theta)\eta_{i}(I|\theta)$. Accordingly, the self-correlation matrix $C^{ij}(\theta)$ of stochastic variables $I$ is inferior bound by the inverse $M^{ij}(\theta)$ of the self-correlation matrix of the differential forces $\eta_{i}$.

\subsection{Inference theory}
Inference theory can be described as the problem of deciding how well a set of outcomes $\mathcal{I}=\left\{I^{(1)},I^{(2)},...I^{(m)}\right\} $ obtained from independent measurements fits a proposed distribution function $dp(I|\theta)$ \cite{Fisher}. This question is fully equivalent to infer the values of control parameters $\theta$ from this last experimental information. To make inferences about control parameters, one employs \textit{estimators} $\hat{\theta}^{\alpha}=\theta^{\alpha}(\mathcal{I})$, that is, functions on the outcomes $\mathcal{I}\in\mathcal{M}^{m}_{\theta}$,
where $\mathcal{M}^{m}_{\theta}=\mathcal{M}_{\theta}\otimes \mathcal{M}_{\theta}\ldots\otimes\mathcal{M}_{\theta}$ (m-times the external product of the statistical manifold $\mathcal{M}_{\theta}$). The values of these functions pretend to be the best guess for $\theta^{\alpha}$.

Let us admit that the probability density $\rho(I|\theta)$ is everywhere differentiable and finite on the statistical manifold $\mathcal{P}$ of control parameters $\theta$. Let us start introducing the statistical expectation values $\left\langle A\left( \mathcal{I}|\theta \right)\right\rangle$ as follows:
\begin{equation}  \label{inf.exp.value}
\left\langle A\left( \mathcal{I}|\theta \right)\right\rangle=\int_{\mathcal{M%
}^{m}_{\theta}}A\left( \mathcal{I}|\theta \right)\varrho \left( \mathcal{I}|
\theta \right)d\mathcal{I},
\end{equation}
where $d\mathcal{I}=dI^{(1)}dI^{(2)}...dI^{(m)}$ and $\varrho \left(\mathcal{I}| \theta \right)$ is the so-called \textit{likelihood function}:
\begin{equation}  \label{likelihood}
\varrho \left( \mathcal{I}| \theta
\right)=\prod^{m}_{i=1}\rho(I^{(i)}|\theta).
\end{equation}
Taking the partial derivative $\partial_{\alpha}=\partial /\partial \theta^{\alpha }$ of Eq.(\ref{inf.exp.value}), one obtains the following mathematical identity:
\begin{equation}  \label{Ident.type2}
\left\langle \partial_{\alpha} A\left( \mathcal{I}| \theta
\right)  \right\rangle -\partial_{\alpha}
\left\langle A\left( \left. \mathcal{I}\right\vert \theta \right)
\right\rangle =\left\langle A\left( \left.
\mathcal{I}\right\vert \theta \right) \upsilon_{\alpha }\left( \left.
\mathcal{I}\right\vert \theta \right) \right\rangle,
\end{equation}
where:
\begin{equation}\label{score.vector}
\hat{\upsilon} _{\alpha }=\upsilon _{\alpha }\left( \left. \mathcal{I}%
\right\vert \theta \right) =-\frac{\partial }{\partial \theta^{\alpha }}\log
\rho \left( \mathcal{I}\left\vert \theta \right. \right)
\end{equation}
are the components of the \textit{score vector} $\hat{\upsilon} =\left\{\hat{\upsilon} _{\alpha }\right\} $. Substituting $A\left( \left. \mathcal{I}\right\vert \theta \right) =1$ into Eq.(\ref{Ident.type2}), one arrives at the vanishing of the expectation values of the score vector components:
\begin{equation}
\left\langle \upsilon _{\alpha }\left( \left. \mathcal{I}\right\vert \theta
\right) \right\rangle =0.  \label{equi.score}
\end{equation}%
Let us consider now any \textit{unbiased estimator} $A(\mathcal{I}|\theta)=\theta^{\alpha }(\mathcal{I})$ of the parameter $\theta^{\alpha}$, $\left\langle\theta^{\alpha }(\mathcal{I})\right\rangle=\theta^{\alpha}$, as well as the score vector component $A(\mathcal{I}| \theta)=\upsilon_{\beta }(\mathcal{I}|\theta)$. Substituting these quantities into identity (\ref{Ident.type2}), it is possible to obtain the following results:
\begin{equation}
\left\langle \delta \theta^{\alpha }\left( \mathcal{I}\right) \upsilon
_{\beta }\left(\mathcal{I}\left\vert \theta \right. \right) \right\rangle
=-\delta _{\beta }^{\alpha },  \label{fund.infe.theo}
\end{equation}
\begin{equation}
\left\langle \partial_{\beta} \upsilon _{\alpha }\left(  \mathcal{I}
| \theta \right)\right\rangle
=\left\langle \upsilon_{\alpha }\left( \left. \mathcal{I}\right\vert
\theta \right)  \upsilon _{\beta }\left( \left. \mathcal{I}\right\vert
\theta \right) \right\rangle .  \label{comple.infe.theo}
\end{equation}
It is easy to realize that the identities (\ref{equi.score}) and (\ref{comple.infe.theo}) can be regarded as the stationary and stability conditions of the known method of \emph{maximum likelihood estimators} \cite{Fisher}. According to this method, the best values of the parameters $\theta$ should maximize the logarithm of the \textit{likelihood function} $\varrho \left( \mathcal{I}| \theta
\right)$ for a given set of outcomes $\mathcal{I}$. Such an exigence leads to the following \textit{stationary and stability conditions}:
\begin{equation}
-\frac{\partial }{\partial \theta ^{\alpha }}\log\varrho \left( \mathcal{I}|
\bar{\theta} \right) =0,-\frac{\partial^{2} }{\partial \theta ^{\alpha }\partial
\theta ^{\beta }}\log\varrho \left( \mathcal{I}| \bar{\theta} \right)\succ 0,
\end{equation}
which should be solved to obtain the \emph{maximum likelihood estimators} $\hat{\theta}^{\alpha}_{mle}=\bar{\theta}^{\alpha}(\mathcal{I})$. On the other hand, the identity (\ref{fund.infe.theo}) also suggests the statistical complementarity between the estimator $\hat{\theta}^{\alpha}=\theta^{\alpha}(\mathcal{I})$ and its conjugated score vector component $\upsilon_{\alpha}$. Using the Schwartz inequality, one obtains the following uncertainty-like inequality:
\begin{equation}
\Delta \hat{\theta}^{\alpha}\Delta\upsilon_{\alpha}\geq 1.
\end{equation}
This result can be easily improved introducing the inverse matrix $g^{\alpha\beta}(\theta)$ of the self-correlation matrix $g_{\alpha\beta}(\theta)=\left\langle\upsilon_{\alpha}(\mathcal{I}|\theta)\upsilon_{\beta}(\mathcal{I}|\theta)
\right\rangle$ and the auxiliary quantity $X^{\alpha}=\delta\hat{\theta}^{\alpha}-g^{\alpha\beta}\upsilon_{\beta}$. Thus, one can compose the positive definite form:
\begin{equation}
\left\langle\left(\lambda_{\alpha}X^{\alpha}\right)^{2}\right\rangle=\left\langle
X^{\alpha}X^{\beta}\right\rangle\lambda_{\alpha}\lambda_{\beta}\geq
0,
\end{equation}
which leads to the positive definition of the matrix:
\begin{equation}\label{CramerRao}
\left\langle\delta
\hat{\theta}^{\alpha}\delta\hat{\theta}^{\beta}\right\rangle-g^{\alpha\beta}(\theta)\succeq
0.
\end{equation}
This last inequality is the famous \textit{Cramer-Rao theorem} of inference theory \cite{Fisher,Rao} that imposes an inferior bound to the efficiency of unbiased estimators $\hat{\theta}^{\alpha}$, where the self-correlation matrix $g_{\alpha\beta}(\theta)$:
\begin{equation}\label{Fisher}
g_{\alpha\beta}(\theta)=\left\langle\upsilon_{\alpha}(\mathcal{I}|\theta)\upsilon_{\beta}(\mathcal{I}|\theta)
\right\rangle
\end{equation}
is the Fisher's information matrix referred to in the introductory section.

\subsection{Analogy between inference theory and fluctuation theory}
Fluctuation theory and inference theory provide two different but \emph{complementary} characterizations for a given parametric family of distribution functions $dp(I|\theta)$. Formally, these two statistical frameworks can be regarded as \emph{dual counterpart approaches} because of the great analogy among their main definitions and theorems, as clearly evidenced in table \ref{Duality}. To simplify the notation, it was introduced here the gradient operators $\partial_{i}\rightarrow\mathbf{\nabla}_{I}$ and $\partial_{\alpha}\rightarrow\mathbf{\nabla}_{\theta}$, the diadic products $\mathbf{A}\cdot \mathbf{B}=A_{i}B_{j}\mathbf{e}^{i}\cdot \mathbf{e}^{j}$ and $\mathbf{\xi}\cdot \mathbf{\psi}=\xi_{\alpha}\psi_{\beta}\mathbf{\epsilon}^{\alpha}\cdot \mathbf{\epsilon}^{\beta}$ and the Kroneker delta
$\delta^{i}_{j}\rightarrow\mathbf{1}_{I}$ and $\delta^{\alpha}_{\beta}\rightarrow\mathbf{1}_{\theta}$.

\begin{table}[tbp] \centering
\begin{tabular}{|c|c|}
\hline\hline \textbf{Inference theory} & \textbf{Fluctuation theory}
\\ \hline\hline
\multicolumn{1}{|c|}{$\mathbf{\upsilon }\left( \mathcal{I}|\theta \right) =-%
\mathbf{\nabla }_{\theta }\log \rho \left( \mathcal{I}|\theta \right) $} & $%
\mathbf{\eta }\left( I|\theta \right) =-\mathbf{\nabla }_{I}\log
\rho \left(
I|\theta \right) $ \\
$\left\langle \mathbf{\upsilon }\left( \mathcal{I}|\theta \right)
\right\rangle =0$ & $\left\langle \mathbf{\eta }\left( I|\theta
\right)
\right\rangle =0$ \\
$\left\langle \mathbf{\upsilon }\left( \mathcal{I}|\theta \right)
\cdot \delta \mathbf{\hat{\theta}}\right\rangle
=-\mathbf{1}_{\theta}$ & $\left\langle
\mathbf{\eta }\left( I|\theta \right) \cdot \delta \mathbf{I}\right\rangle =%
\mathbf{1}_{I}$ \\
$\left\langle \mathbf{\nabla }_{\theta }\cdot \mathbf{\upsilon
}\left(
\mathcal{I}|\theta \right) \right\rangle =\left\langle \mathbf{\upsilon }%
\left( \mathcal{I}|\theta \right) \cdot \mathbf{\upsilon }\left( \mathcal{I}%
|\theta \right) \right\rangle $ & $\left\langle \mathbf{\nabla
}_{I}\cdot
\mathbf{\eta }\left( I|\theta \right) \right\rangle =\left\langle \mathbf{%
\eta }\left( I|\theta \right) \cdot \mathbf{\eta }\left( I|\theta
\right) \right\rangle $ \\ \hline
\end{tabular}
\caption{Analogy between inference theory and fluctuation
theory.}\label{Duality}
\end{table}

Remarkably, the analogy between fluctuation theory and inference theory is uncomplete in regard to their respective \emph{geometric features}. The parametric family $dp(I|\theta)$ is expressed in the representations $\mathcal{R}_{I}$ and $\mathcal{R}_{\theta}$ of the statistical manifolds $\mathcal{M}_{\theta}$ and $\mathcal{P}$, respectively. Equivalently, the same parametric family can be also rewritten using  the representations $\mathcal{R}_{\Theta}$ and $\mathcal{R}_{\nu}$ of the manifolds $\mathcal{M}_{\theta}$ and $\mathcal{P}$, which implies the consideration of the coordinate changes $\Theta(I):\mathcal{R}_{I}\rightarrow\mathcal{R}_{\Theta}$ and $\nu(\theta):\mathcal{R}_{\theta}\rightarrow\mathcal{R}_{\nu}$. Under these parametric changes, the Fisher's inference matrix (\ref{Fisher}) behaves as the components of a second rank covariant tensor:
\begin{equation}
g_{\gamma\delta}(\nu)=\frac{\partial \theta^{\alpha}}{\partial\nu^{\gamma}}\frac{\partial \theta^{\beta}}{\partial\nu^{\delta}}g_{\alpha\beta}(\theta).
\end{equation}
The existence of these last transformation rules guarantees the invariance of the inference distance notion (\ref{inf.dist}). Thus, the statistical manifold $\mathcal{P}$ of control parameters $\theta$ can be endowed of a \emph{Riemannian structure}. The relevance of the distance notion (\ref{inf.dist}) can be understood considering the asymptotic expression of the distribution function of the \emph{efficient unbiased estimators} $\hat{\theta}_{eff}(\mathcal{I})$:
\begin{equation}\label{asymp.dist}
dQ^{m}(\vartheta|\theta)=\int_{\mathcal{M}^{m}_{\theta}}\delta\left[\vartheta-\hat{\theta}_{eff}
(\mathcal{I})\right]\varrho(\mathcal{I}|\theta)d\mathcal{I}
\end{equation}
when the number of outcomes $m$ is sufficiently large. Since $g_{\alpha\beta}(\theta)\propto m$, one can obtain the following approximation formula:
\begin{equation}\label{asympt.inf}
dQ^{m}(\vartheta|\theta)\simeq\exp\left[-\frac{1}{2}
g_{\alpha\beta}(\theta)\Delta\vartheta^{\alpha}\Delta\vartheta^{\beta}\right]\sqrt{\left|\frac{
g_{\alpha\beta}(\theta)}{2\pi}\right|}d\vartheta,
\end{equation}
where $\Delta\vartheta^{\alpha}=\vartheta^{\alpha}-\theta^{\alpha}$. Accordingly, the distance notion (\ref{inf.dist}) provides
the \textit{distinguishing probability} between two close distribution functions of the parametric family $dp(I|\theta)$ through the inferential procedure.

The analogy between fluctuation theory and inference theory strongly suggests the existence of \emph{a counterpart approach of inference geometry in the framework of fluctuation theory}, that is, the existence of a Riemannian distance notion (\ref{fluct.dist}) to characterize the statistical separation between close points $I$ and $I+dI\in\mathcal{M}_{\theta}$. Unfortunately, the underlying analogy is insufficient to introduce the particular expression of the metric tensor $g_{ij}(I|\theta)$. For example, it is easy to check that the fluctuation theorems (\ref{equi.cond})-(\ref{assoc.fluct}) can be also expressed in the new coordinate representation $\mathcal{R}_{\Theta}$ \cite{Vel.GEO}, as example, the associated fluctuation theorem:
\begin{equation}
\left\langle\partial_{i}\eta_{j}(\Theta|\theta)\right\rangle=\left\langle\eta_{i}(\Theta|\theta)\eta_{j}(\Theta|\theta)\right\rangle,
\end{equation}
where $\partial_{i}=\partial/\partial\Theta^{i}$ and $\eta_{i}(\Theta|\theta)$:
\begin{equation}
\eta_{i}(\Theta|\theta)=-\frac{\partial}{\partial\Theta^{i}}\log\rho(\Theta|\theta).
\end{equation}
However, the self-correlation matrix $\tilde{M}_{ij}(\theta)=\left\langle\eta_{i}(\Theta|\theta)\eta_{j}(\Theta|\theta)\right\rangle$ associated with the new coordinate representation $\mathcal{R}_{\Theta}$ is not related by \emph{local transformation rules} to its counterpart expression $M_{ij}(\theta)=\left\langle\eta_{i}(I|\theta)\eta_{j}(I|\theta)\right\rangle$ in the old coordinate representation $\mathcal{R}_{I}$. The self-correlation matrix of the differential forces $M_{ij}(\theta)=\left\langle\eta_{i}(I|\theta)\eta_{j}(I|\theta)\right\rangle$ is a matrix function defined on the statistical manifold $\mathcal{P}$ of control parameters $\theta$, while the metric tensor $g_{ij}(I|\theta)$ is a \emph{tensorial entity} defined on the statistical manifolds $\mathcal{M}_{\theta}$ and $\mathcal{P}$. As expected, the definition of the metric tensor $g_{ij}(I|\theta)$ cannot involve integral expressions over the manifold $\mathcal{M}_{\theta}$ as the case of inference metric tensor (\ref{Fisher}).

\section{Fluctuation geometry}\label{Postulates}
Fluctuation geometry can be formulated starting from a set of axioms that combine the statistical nature of the manifold $\mathcal{M}_{\theta}$ and the notions of differential geometry. These axioms specify the way to introduce the metric tensor $g_{ij}(I|\theta)$ associated with the parametric family (\ref{DF}). This section is devoted to discuss these axioms and their most direct consequences.

\subsection{Postulates}

\begin{axiom}
The manifold of the stochastic variables $\mathcal{M}_{\theta}$ possesses a \textbf{Riemannian structure}, that is, it is provided
of a \textbf{metric tensor} $g_{ij}\left( I|\theta \right)$ and a \textbf{torsionless  covariant differentiation} $D_{i}$ that obeys
the following constraints:
\begin{equation}
D_{k}g_{ij}\left( I|\theta \right) =0.  \label{Dg}
\end{equation}
\end{axiom}
\begin{definition}
The Riemannian structure on the statistical manifold $\mathcal{M}_{\theta}$ allows to introduce the \textbf{invariant
volume element} as follows:
\begin{equation}  \label{inv.volume}
d\mu(I|\theta)=\sqrt{\left\vert \frac{g_{ij}\left( I|\theta \right) }{2\pi }%
\right\vert }dI,
\end{equation}
where $\left\vert g_{ij}\left( I|\theta \right)\right\vert $ denotes the absolute value of the metric tensor determinant.
\end{definition}
\begin{axiom}
There exist a differentiable scalar function $\mathcal{S}\left(I|\theta \right) $ defined on the statistical manifold $\mathcal{M}_{\theta }$, hereafter referred to as the \textbf{information potential}, whose knowledge determines the distribution function $dp\left(I|\theta \right)$ of the stochastic variables $I\in\mathcal{M}_{\theta}$ as follows:
\begin{equation}
dp\left( I|\theta \right) =\exp \left[ \mathcal{S}\left( I|\theta
\right) \right] d\mu(I|\theta).  \label{EinsteinPostulate}
\end{equation}
\end{axiom}
\begin{definition}
Let us consider an arbitrary curve given in parametric form $I(t)\in\mathcal{M}_{\theta}$ with fixed extreme points $I(t_{1})=P$ and $I(t_{2})=Q$. Adopting the following notation:
\begin{equation}
\dot{I}^{i}=\frac{d I^{i}(t)}{dt},
\end{equation}
the \textbf{length} $\Delta s $ of this curve can be expressed as:
\begin{equation}\label{distance}
\Delta s=\int_{t_{1}}^{t_{2}}\sqrt{g_{ij}\left[ I(t)|\theta \right] \dot{I}%
^{i}\left( t\right) \dot{I}^{j}\left( t\right) }dt.
\end{equation}
\end{definition}
\begin{definition}
It is say that the curve $I(t)\in\mathcal{M}_{\theta}$ exhibits a \textbf{unitary affine parametrization} when its parameter $t$ satisfies the following constraint:
\begin{equation}\label{affine}
g_{ij}[I(t)|\theta]\dot{I}^{i}(t)\dot{I}^{j}(t)=1.
\end{equation}
\end{definition}
\begin{definition}
A \textbf{geodesic} is the curve $I_{g}\left( t\right) $ with minimal length (\ref{distance}) between two fixed arbitrary points $(P,Q)\in\mathcal{M}_{\theta}$. Moreover, the \textbf{distance} $\mathfrak{D}_{\theta}(P,Q)$ between these two points $(P,Q)$ is given by the length of its associated geodesic $I_{g}(t)$:
\begin{equation}\label{distance}
\mathfrak{D}_{\theta}(P,Q)=\int_{t_{1}}^{t_{2}}\sqrt{g_{ij}\left[ I_{g}(t)|\theta \right] \dot{I}_{g}%
^{i}\left( t\right) \dot{I}_{g}^{j}\left( t\right) }dt.
\end{equation}
\end{definition}
\begin{definition}
Let us consider a differentiable curve $I(t)\in\mathcal{M}_{\theta}$ with an unitary affine parametrization. The \textbf{information dissipation} $\Phi(t)$ along the curve $I(t)$ is defined as follows:
\begin{equation}\label{rate}
\Phi(t)=\frac{d\mathcal{S}\left[I(t)|\theta\right]}{dt}.
\end{equation}
\end{definition}
\begin{axiom}
The length $\Delta s$ of any interval $(t_{1},t_{2})$ of an arbitrary \textit{geodesic} $I_{g}(t)\in\mathcal{M}_{\theta}$ with a unitary affine parametrization is given by the negative of the variation of its information dissipation $\Delta\Phi(t)$:
\begin{equation}\label{metric}
\Delta s=-\Delta\Phi(t)=\Phi(t_{1})-\Phi(t_{2}).
\end{equation}
\end{axiom}
\begin{axiom}
If $\mathcal{M}_{\theta}$ is not a closed manifold, $\partial\mathcal{M}_{\theta}\neq\emptyset$, the probability density $\rho(I|\theta)$ associated with distribution function (\ref{EinsteinPostulate}) vanishes with its
first partial derivatives for any point on the boundary $\partial\mathcal{M}_{\theta}$ of the statistical manifold $\mathcal{M}_{\theta}$.
\end{axiom}

\subsection{Analysis of axioms and their direct consequences}
\textbf{Axiom 1} postulates the existence of the metric tensor $g_{ij}(I|\theta)$ defined on the statistical manifold $\mathcal{M}_{\theta}$. Even, this axiom specifies the \emph{Riemannian structure} of the manifold $\mathcal{M}_{\theta}$ starting from the knowledge of the metric tensor $ g_{ij}(I|\theta)$, e.g.: the \emph{covariant differentiation} $D_{i}$ and the \emph{curvature tensor} $R_{ijkl}(I|\theta)$. Equation (\ref{Dg}) is an strong constraint of Riemannian geometry that determines a natural \textit{affine connections} $\Gamma _{ij}^{k}$ for the covariant differentiation $D_{i}$, specifically, the so-called \textit{Levi-Civita connection} \cite{Berger}:
\begin{equation}
\Gamma _{ij}^{k}\left( I|\theta \right) =g^{km}\frac{1}{2}\left( \frac{%
\partial g_{im}}{\partial I^{j}}+\frac{\partial g_{jm}}{\partial I^{i}}-%
\frac{\partial g_{ij}}{\partial I^{m}}\right).  \label{Levi-Civita}
\end{equation}
The knowledge of the affine connections $\Gamma _{ij}^{k}$ allows the introduction of the \textit{curvature tensor} $ R^{l}_{ijk}=R^{l}_{ijk}(I|\theta)$ of the manifold $\mathcal{M}_{\theta}$:
\begin{equation}\label{curvature}
R^{l}_{ijk}=\frac{\partial}{\partial I^{i}}\Gamma^{l}_{jk}-\frac{\partial}{%
\partial I^{j}}\Gamma^{l}_{ik}+\Gamma^{l}_{im}\Gamma^{m}_{jk}-
\Gamma^{l}_{jm}\Gamma^{m}_{ik},
\end{equation}
which is also derived from the knowledge of the metric tensor $g_{ij}(I|\theta)$ and its first and second partial derivatives.

\textbf{Axiom 2} postulates the probabilistic nature of the manifold $\mathcal{M}_{\theta}$, in particular, the existence of the distribution function $dp(I|\theta)$ and the information potential $\mathcal{S}(I|\theta)$. Equivalently, this axiom provides a formal definition for the information potential $\mathcal{S}(I|\theta)$ when one starts from the knowledge of the parametric family (\ref{DF}). The probability density $\rho(I|\theta)$ of the parametric family (\ref{DF}) obeys the transformation rule of a tensorial density:
\begin{equation}\label{tr.density}
\rho(\Theta|\theta)=\rho(I|\theta)\left|\frac{\partial
\Theta}{\partial I}\right|^{-1}
\end{equation}
under coordinate change $\Theta(I):\mathcal{R}_{I}\rightarrow\mathcal{R}_{\Theta}$ of the statistical manifold $\mathcal{M}_{\theta}$. The covariance of the metric tensor $g_{ij}\left( I|\theta \right)$:
\begin{equation}
g_{ij}\left( \Theta|\theta \right)=\frac{\partial I^{m}}{\partial\Theta^{i}}%
\frac{\partial I^{n}}{\partial\Theta^{j}}g_{mn}\left( I|\theta
\right)
\end{equation}
implies that the pre-factor of the invariant volume element (\ref{inv.volume}) also behaves as a tensorial density:
\begin{equation}
\sqrt{\left\vert \frac{g_{ij}\left( \Theta|\theta \right) }{2\pi }%
\right\vert }=\sqrt{\left\vert \frac{g_{ij}\left( I|\theta \right) }{2\pi }%
\right\vert }\left|\frac{\partial \Theta}{\partial I}\right|^{-1}.
\end{equation}
Admitting that the metric tensor determinant $\left|g_{ij}(I|\theta)\right|$ is non-vanishing everywhere, it is possible to introduce \textit{probability weight} $\omega(I|\theta)$:
\begin{equation}\label{scalar.weight}
\omega(I|\theta)=\rho(I|\theta)\sqrt{\left\vert 2\pi g^{ij}\left( I|\theta \right)\right\vert },
\end{equation}
which represents a \emph{scalar function} defined on the manifold $\mathcal{M}_{\theta}$. Since the statistical manifold $\mathcal{M}_{\theta}$ possesses a Riemannian structure, the integration over the usual volume element $dI$ can be replaced by the invariant volume element $d\mu(I|\theta)$. This consideration allows to rephrase the parametric family (\ref{DF}) in the following equivalent representation:
\begin{equation}
dp(I|\theta)=\omega(I|\theta)d\mu(I|\theta),
\end{equation}
which explicitly exhibits the covariance of the distribution function under the coordinate reparametrizations of the manifold $\mathcal{M}_{\theta}$. The information potential $\mathcal{S}(I|\theta)$ is defined by the logarithm of the probability weight $\omega(I|\theta)$:
\begin{equation}
\mathcal{S}(I|\theta)=\log \omega(I|\theta),
\end{equation}
which also represents a scalar function defined on the statistical manifold $\mathcal{M}_{\theta}$. As discussed in section \ref{Relevance}, the negative of the information potential $\mathcal{S}(I|\theta)$ can be regarded as a \emph{local invariant measure of the information content} in the framework of information theory. Additionally, $\mathcal{S}(I|\theta)$ can be identified with the \emph{scalar entropy of a closed system} in the framework of classical fluctuation theory \cite{Vel.GEO}. Given the probability density $\rho(I|\theta)$, the information potential $\mathcal{S}(I|\theta)$ depends on the metric tensor $g_{ij}(I|\theta)$ of the statistical manifold $\mathcal{M}_{\theta}$. \textbf{Axiom 1} postulates the existence of this tensor, but its specific definition is still arbitrary. \textbf{Axiom 3} eliminates such an ambiguity. In fact, it establishes a direct connection between the distance notion (\ref{fluct.dist}) and the information dissipation (\ref{rate}), or equivalently, between the metric tensor $g_{ij}(I|\theta)$ and the information potential $\mathcal{S}(I|\theta)$.

\begin{theorem}
The metric tensor $g_{ij}(I|\theta)$ can be identified with the negative of the \textbf{covariant Hessian} $\mathcal{H}_{ij}(I|\theta)$ of the information potential $\mathcal{S}\left( I|\theta \right)$:
\begin{equation}
g_{ij}\left( I|\theta \right)
=-\mathcal{H}_{ij}(I|\theta)=-D_{i}D_{j}\mathcal{S}\left( I|\theta
\right) . \label{cov.EH}
\end{equation}
\end{theorem}
\begin{corollary}\label{concavity}
\textit{The information potential} $\mathcal{S}(I|\theta)$ \textit{is locally concave everywhere} and the metric tensor $g_{ij}(I|\theta)$ is \textit{positive definite} on the statistical manifold $\mathcal{M}_{\theta}$.
\end{corollary}
\begin{proof}
The searching of the curve with minimal length (\ref{distance}) between two arbitrary points $(P,Q)$ is a variational problem that leads to the so-called \emph{geodesic differential equations} \cite{Berger}:
\begin{equation} \label{geodesic}
\dot{I}_{g}^{k}(t)D_{k}\dot{I}_{g}^{i}(t)=\ddot{I}_{g}^{i}(t)+
\Gamma^{i}_{mn}\left[I_{g}(t)|\theta\right]\dot{I}_{g}^{m}(t)\dot{I}%
_{g}^{n}(t)=0,
\end{equation}
which describes the geodesics $I_{g}(t)$ with a unitary affine parametrization. Equations (\ref{rate}) and (\ref{metric}) can be rephrased as follows:
\begin{equation}\label{alter.post3}
\Delta s=-\Delta\Phi(s)\rightarrow\frac{d\Phi(s)}{ds}=
\frac{d^{2}\mathcal{S}}{ds^{2}}=-1.
\end{equation}
Considering the geodesic differential equations (\ref{geodesic}), the constraint (\ref{alter.post3}) is rewritten as:
\begin{eqnarray}  \label{razonamiento}
\frac{d^{2}\mathcal{S}}{ds^{2}} = \ddot{I}_{g}^{k}\frac{\partial
\mathcal{S}}{\partial I^{k}}+\dot{I}_{g}^{i}\dot{I}_{g}^{j}\frac{%
\partial^{2} \mathcal{S}}{\partial I^{i}\partial I^{j}}
=\dot{I}_{g}^{i}\dot{I}_{g}^{j}\left\{\frac{\partial^{2} \mathcal{S}}{%
\partial I^{i}\partial I^{j}}-\Gamma^{k}_{ij}\frac{\partial \mathcal{S}}{%
\partial I^{k}}\right\},
\end{eqnarray}
which involves the \textit{covariant Hessian} $\mathcal{H}_{ij}$:
\begin{equation}\label{Hessian}
\mathcal{H}_{ij}=D_{i}D_{j}\mathcal{S}=\frac{\partial^{2} \mathcal{S}}{%
\partial I^{i}\partial I^{j}}-\Gamma^{k}_{ij}\frac{\partial \mathcal{S}}{%
\partial I^{k}}.
\end{equation}
Combining equations (\ref{razonamiento})-(\ref{Hessian}) and the constraint (\ref{affine}), one obtains the following expression:
\begin{equation}
(g_{ij}+\mathcal{H}_{ij})\dot{I}_{g}^{i}\dot{I}_{g}^{j}=0.
\end{equation}
Its covariant character leads to Eq.(\ref{cov.EH}). \textbf{Corollary \ref{concavity}}, that is, the concave character of the information potential $\mathcal{S}(I|\theta)$ and the positive definition of the metric tensor $g_{ij}(I|\theta)$ are direct consequences of equation (\ref{alter.post3}).
\end{proof}
\begin{corollary}
The metric tensor $g_{ij}=g_{ij}\left( I|\theta \right) $ can be obtained from the probability density $\rho=\rho\left( I|\theta \right)$ through the following set of covariant second-order partial differential equations:
\begin{equation}  \label{cov.equation}
g_{ij}=-\frac{\partial^{2}\log\rho}{\partial I^{i}\partial I^{j}}+
\Gamma^{k}_{ij}\frac{\partial\log\rho}{\partial I^{k}} +\frac{\partial}{\partial I^{i}}\Gamma^{k}_{jk}-\Gamma^{k}_{ij}\Gamma^{l}_{kl}.
\end{equation}
The admissible solutions for the metric tensor $g_{ij}$ should be finite and differentiable everywhere, including also, the boundary $\partial\mathcal{M}_{\theta}$ of the statistical manifold $\mathcal{M}_{\theta}$.
\end{corollary}
\begin{proof}
Expression (\ref{cov.equation}) is derived from equation (\ref{cov.EH}) rewriting the information potential as $\mathcal{S}\left( I|\theta \right)\equiv\log\rho\left( I|\theta \right)-\log\sqrt{|g_{ij}\left( I|\theta \right)/2\pi|}$ and considering the following identity:
\begin{equation}
\Gamma^{j}_{ij}\left( I|\theta \right)\equiv\frac{\partial\log\sqrt{|g_{ij}\left( I|\theta \right)/2\pi|}}{\partial I^{i}},
\end{equation}
which is a known property of the Levi-Civita connection (\ref{Levi-Civita}).
\end{proof}

\textbf{Axiom 4} talks about the asymptotic behavior of the distribution function (\ref{EinsteinPostulate}) for any point $I_{b}$ on the boundary $\partial\mathcal{M}_{\theta}$:
\begin{equation}  \label{boundary.cond}
\lim_{I\rightarrow I_{b}}\rho(I|\theta)=\lim_{I\rightarrow I_{b}}\frac{%
\partial}{\partial I^{i}}\rho(I|\theta)=0.
\end{equation}
These last conditions are necessary to obtain the general fluctuation theorems (\ref{equi.cond})-(\ref{assoc.fluct}) reviewed in the previous section. Moreover, this axiom will be employed to analyze the character of stationary points (maxima and minima) of the information potential $\mathcal{S}(I|\theta)$.

\begin{remark}
The boundary conditions (\ref{boundary.cond}) are independent from the admissible coordinate representation $\mathcal{R}_{I}$ of the statistical manifold $\mathcal{M}_{\theta}$. Moreover, the probability weight $\omega(I|\theta)$ vanishes on the boundary $\partial \mathcal{M}_{\theta}$ of the manifold $\mathcal{M}_{\theta}$.
\end{remark}
\begin{proof}
This remark is a direct consequence of the transformation rule of the probability density (\ref{tr.density}) as well as the ones associated with its partial derivatives:
\begin{equation}  \label{drho.tr}
\frac{\partial \rho \left( \Theta |\theta \right) }{\partial \Theta ^{i}} =
\frac{\partial I^{j}}{\partial \Theta ^{i}}\left\{ \frac{\partial \rho
\left( I|\theta \right) }{\partial I^{j}}-\rho \left( I|\theta \right) \frac{%
\partial }{\partial I^{j}}\log \left\vert \frac{\partial \Theta }{\partial I}%
\right\vert \right\} \left\vert \frac{\partial \Theta }{\partial I}%
\right\vert ^{-1}
\end{equation}
under a coordinate change $\Theta(I):\mathcal{R}_{I}\rightarrow\mathcal{R}_{\Theta}$ with Jacobian $\left|\partial\Theta/\partial I \right|$ finite and differentiable everywhere. Since the metric tensor determinant $\left|g_{ij}(I|\theta)\right|$ is non-vanishing everywhere, \textbf{Axiom 4} directly implies the vanishing of the probability weight $\omega(I|\theta)$ on the boundary $\partial \mathcal{M}_{\theta}$ of the statistical manifold $\mathcal{M}_{\theta}$.
\end{proof}

\section{Geometric reinterpretation of the statistical description}\label{Consequences}

The question about the \emph{existence and uniqueness} of solutions obtained from the problem (\ref{cov.equation}) cannot be fully analyzed in this work because of its complexity. This section is devoted to discuss some consequences derived from the existence of a given particular solution $g_{ij}(I|\theta)$.

\subsection{Gaussian representation}
\begin{definition}
The covariant components of the \textbf{gradiental vector} $\psi _{i}\left( I|\theta \right) $ are defined from the information potential $\mathcal{S}\left( I|\theta \right)$ as follows:
\begin{equation}
\psi _{i}\left( I|\theta \right) =-D_{i}\mathcal{S}\left( I|\theta
\right) \equiv-\partial \mathcal{S}\left( I|\theta \right) /\partial
I^{i}.  \label{cov.DGF}
\end{equation}
Using the metric tensor $g^{ij}(I|\theta)$, it is possible to obtain its contravariant counterpart $\psi^{i}(I|\theta)$:
\begin{equation}
\psi ^{i}\left( I|\theta \right) =g^{ij}\left( I|\theta \right) \psi
_{j}\left( I|\theta \right),
\end{equation}
as well as its the square norm $\psi ^{2}=\psi ^{2}(I|\theta)$:
\begin{equation}
\psi ^{2}(I|\theta)=\psi ^{i}\left( I|\theta \right)\psi _{i}\left(
I|\theta \right).
\end{equation}
\end{definition}
\begin{theorem}
The information potential $\mathcal{S}(I|\theta )$ can be expressed in terms of the square norm of the gradiental vector as follows:
\begin{equation}
\mathcal{S}(I|\theta )=\mathcal{P}(\theta )-\frac{1}{2}%
\psi^{2}(I|\theta ), \label{Sdecomposition}
\end{equation}
where $\mathcal{P}(\theta )$ is a certain function on control parameters $\theta$, which is hereafter referred to as the \textbf{gaussian
potential}.
\end{theorem}
\begin{proof}
Let us introduce the scalar function $\mathcal{P}(I|\theta )$:
\begin{equation}\label{PP}
\mathcal{P}(I|\theta )=\mathcal{S}(I|\theta )+\frac{1}{2}%
g^{ij}(I|\theta )\psi _{i}(I|\theta )\psi _{j}(I|\theta ).
\end{equation}%
It is easy to verify that its covariant derivatives:
\begin{eqnarray}
&&D_{k}\mathcal{P}(I|\theta )=D_{k}\mathcal{S}(I|\theta )+ \frac{1}{2}\left\{
\psi _{i}(I|\theta )\psi _{j}(I|\theta
)D_{k}g^{ij}(I|\theta )+\right.   \\
&&\left. +g^{ij}(I|\theta )\left[ \psi _{i}(I|\theta )D_{k}\psi
_{j}(I|\theta )+\psi _{j}(I|\theta )D_{k}\psi _{i}(I|\theta )\right]
\right\} \nonumber
\end{eqnarray}
vanish as direct consequences of the metric tensor properties (\ref{Dg}) and (\ref{cov.EH}), as well as definition
(\ref{cov.DGF}) of the gradiental vector $\psi _{i}(I|\theta )$. Since the covariant derivatives of any scalar
function are given by the usual partial derivatives:
\begin{equation}
D_{k}\mathcal{P}(I|\theta )=\frac{\partial }{\partial
I^{k}}\mathcal{P} (I|\theta )=0,
\end{equation}%
the scalar function $\mathcal{P}(I|\theta )$ can only depend on the control parameters $\theta$:
\begin{equation}
\mathcal{P}(I|\theta )\equiv \mathcal{P}(\theta ).
\end{equation}
Mathematically speaking, the scalar function (\ref{PP}) can be regarded as a first integral of the set of covariant differential equations (\ref{cov.equation}).
\end{proof}
\begin{corollary}\label{entropy.planck}
The value of information potential $\mathcal{S}(I|\theta )$ at all its
extreme points derived from the stationary condition:
\begin{equation}\label{global.stat}
\psi^{2}(\bar{I}|\theta)=0
\end{equation}
is exactly given by the gaussian potential
$\mathcal{P}(\theta)$.
\end{corollary}
\begin{corollary}
The distribution function (\ref{EinsteinPostulate}) admits the following \textbf{gaussian representation}:
\begin{equation}
dp(I|\theta )=\frac{1}{\mathcal{Z}(\theta )}\exp \left[ -\frac{1}{2}%
\psi^{2} (I|\theta )\right] d\mu (I|\theta ). \label{universal}
\end{equation}
Here, the factor $\mathcal{Z}(\theta )$ is related to the gaussian potential $\mathcal{P}(\theta )$ as follows:
\begin{equation}\label{gaussianPlanck}
\mathcal{P}(\theta )=-\log \mathcal{Z}(\theta ),
\end{equation}
which is hereafter referred to as the \textbf{gaussian partition function}.
\end{corollary}

\subsection{Maximum and completeness theorems}
\begin{theorem}
The information potential $\mathcal{S}(I|\theta)$ exhibits a unique stationary point $\bar{I}$ in the statistical manifold $\mathcal{M}_{\theta}$, which corresponds to its global maximum.
\end{theorem}
\begin{proof}
The information potential $\mathcal{S}(I\theta)$ should exhibit at least a stationary point $\bar{I}$
where takes place the stationary condition (\ref{global.stat}). This conclusion follows from the vanishing of the scalar weight of distribution function:
\begin{equation}
\omega(I|\theta)=\exp\left[\mathcal{S}(I|\theta)\right]
\end{equation}
on the boundary $\partial\mathcal{M}_{\theta}$, as well as its character nonnegative, finite and differentiable on the simply connected manifold $\mathcal{M}_{\theta}$. Since the information potential $\mathcal{S}(I\theta)$ is a concave function, its stationary points can only correspond to local maxima. Let us suppose the existence of at least two stationary points $\bar{I}_{1}$ and $\bar{I}_{2}$ as well as the geodesic $I_{g}(t)$ that connects these points. According to constraint (\ref{alter.post3}), the information dissipation $\Phi(t)$ is a monotonous function along the curve $I_{g}(t)$. Therefore, $\Phi(t)$ should exhibit different values at the stationary points $\bar{I}_{1}$ and $\bar{I}_{2}$, which is absurdum since the information dissipation $\Phi(t)$ identically vanishes for any stationary point of the information potential $\mathcal{S}(I|\theta)$:
\begin{equation}
\Phi(t)=-\dot{I}^{i}(t)\psi_{i}[I(t)|\theta].
\end{equation}
Consequently, there exist only one stationary point that corresponds with the global maximum of the information potential
$\mathcal{S}(I|\theta)$.
\end{proof}
\begin{theorem}\label{entropy.distance}
Any hyper-surface of constant information potential $\mathcal{S}(I|\theta)$ is just the boundary of a $n$-dimensional sphere $S^{n}(\bar{I},\ell)\subset \mathcal{M}_{\theta}$ centered at the point $\bar{I}$ with global maximum information potential, where $n$ is the dimension of the manifold $\mathcal{M}_{\theta}$. Moreover, the information potential $\mathcal{S}$ depends on the radius $\ell$ of this $n$-dimensional sphere as follows:
\begin{equation}\label{E1}
\mathcal{S}=\mathcal{P}(\theta)-\frac{1}{2}\ell^{2}.
\end{equation}
\end{theorem}
\begin{proof}
By definition, the vector field $\upsilon^{i}(I|\theta)$:
\begin{equation}\label{unitary}
\upsilon^{i}(I|\theta)=\frac{\psi^{i}(I|\theta)}{\psi(I|\theta)}
\end{equation}
is the unitary normal vector of the hyper-surface with constant information potential $\mathcal{S}(I|\theta)$. It is easy to verify that the vector field $\upsilon^{i}(I|\theta)$ obeys the geodesic equations (\ref{geodesic}):
\begin{equation}
\upsilon^{k}(I|\theta)D_{k}\upsilon^{i}(I|\theta)=\frac{\upsilon^{k}
(I|\theta)}{\psi(I|\theta)}\left[\delta^{i}_{k}-\upsilon^{i}(I|\theta)\upsilon_{k}(I|\theta)\right]=0.
\end{equation}
Hence, $\upsilon^{i}(I|\theta)$ can be regarded as the tangent vector:
\begin{equation}\label{velocity}
\frac{dI^{i}_{g}(s|\mathbf{e})}{ds}=\upsilon^{i}[I_{g}(s|\mathbf{e})|\theta]
\end{equation}
of geodesic family $I_{g}(s|\mathbf{e})$ with unitary affine parametrization centered at the point $\bar{I}$ with maximum information potential $\mathcal{S}(I|\theta)$, $I_{g}(s=0|\mathbf{e})=\bar{I}$. Moreover, the constant unitary vector $\mathbf{e}$ parameterizes geodesics with different directions at the origin, $\dot{I}_{g}(s=0|\mathbf{e})=\mathbf{e}$ . The information dissipation $\Phi(s|\mathbf{e})$ along any of these geodesics is given by the negative of the norm of the gradiental vector:
\begin{equation}
\Phi(s|\mathbf{e})=-\frac{dI^{i}(s|\mathbf{e})}{ds}\psi_{i}\left[I_{g}(s|\mathbf{e})|\theta\right]=-\psi[I_{g}(s|\mathbf{e})|\theta].
\end{equation}
Considering equation (\ref{metric}), the norm $\psi(I|\theta)$ can be related to the length $\Delta s$ of the geodesic that connects an arbitrary point $I$ with the point $\bar{I}$ with maximum information potential, that is, the \textit{distance} $\mathfrak{D}_{\theta}(I,\bar{I})$ between the points $I$ and $\bar{I}$:
\begin{equation}\label{radius}
\psi(I|\theta)=\mathfrak{D}_{\theta}(I,\bar{I}).
\end{equation}
According to the gaussian decomposition (\ref{Sdecomposition}), the hyper-surface with constant information potential $\mathcal{S}(I|\theta)$ is also the hyper-surface where the norm of gradiental generalized forces $\psi(I|\theta)$ is kept constant, that is, the boundary of a $n$-dimensional sphere $S^{n}(\bar{I},\ell)$ centered at the point $\bar{I}$ with maximum information potential.
\end{proof}
\begin{corollary}
The distribution function (\ref{EinsteinPostulate}) can be expressed in the following \textbf{Riemannian gaussian representation}:
\begin{equation}\label{UGR}
dp(I|\theta)=\frac{1}{\mathcal{Z}(\theta)}\exp\left[-\frac{1}{2}\ell^{2}(I)\right]d\mu(I|\theta),
\end{equation}
where $\ell(I)=\mathfrak{D}_{\theta}(I,\bar{I})$ is the separation distance between the points $\bar{I}$ and $I$. Consequently, the knowledge of the metric tensor $g_{ij}(I|\theta)$ and the point $\bar{I}$ with maximum information potential $\mathcal{S}(I|\theta)$ fully determines the distribution function $dp(I|\theta)$.
\end{corollary}
\begin{proof}
Riemannian gaussian representation (\ref{UGR}) is a direct consequence of replacing equation (\ref{E1}) into equation (\ref{EinsteinPostulate}). The radius $\ell$ of the $n$-dimensional sphere $S^{n}(\bar{I},\ell)$ referred to in \textbf{Theorem \ref{entropy.distance}} and the invariant volume element $d\mu(I|\theta)$ are purely \emph{geometric notions} derived from the knowledge of the metric tensor $g_{ij}(I|\theta)$ and the point $\bar{I}$ with maximum information potential $\mathcal{S}(I|\theta)$. Equation (\ref{UGR}) evidences that all the statistical description associated with the distribution function (\ref{DF}) can be rephrased in terms of geometric notions derived from the Riemannian structure of the manifold $\mathcal{M}_{\theta}$.
\end{proof}
\begin{corollary}
For points $I$ close to the point $\bar{I}$ with maximum information potential $\mathcal{S}(I|\theta)$, the distribution function (\ref{EinsteinPostulate}) admits the following gaussian approximation:
\begin{equation}\label{gaussian.fluct}
dp(I|\theta)\simeq\exp\left[-\frac{1}{2}g_{ij}(\bar{I}|\theta)\Delta I^{i}\Delta I^{i}\right]\sqrt{\left|\frac{g_{ij}(\bar{I}|\theta)}{2\pi}\right|}dI.
\end{equation}
\end{corollary}
\begin{proof}
The separation distance $\ell(I)=\mathfrak{D}_{\theta}(I,\bar{I})$ can be approximated as follows:
\begin{equation}
\ell^{2}(I)\simeq g_{ij}(\bar{I}|\theta)\Delta I^{i}\Delta I^{i}.
\end{equation}
This last expression can be directly obtained from definition (\ref{distance}), where $\Delta I^{i}=I^{i}-\bar{I}^{i}$. In this approximation level, the normalization condition implies the following estimation for gaussian partition function $\mathcal{Z}(\theta)\simeq1$.
\end{proof}

\section{Application examples}\label{Examples}

\subsection{Fluctuation geometry of an one-dimensional statistical manifold $\mathcal{M}_{\theta}$}

Let $dp(I|\theta)$ be a generic parametric family defined on an one-dimensional manifold $\mathcal{M}_{\theta}$. Let us also consider that the admissible values of the stochastic variable in the coordinate representation $\mathcal{R}_{I}$ belong to a certain real subset $(I_{min},I_{max})\subset \mathbb{R}$. Due to its general multidimensional character, the statistical manifold $\mathcal{P}$ of control parameters $\theta$ could be a flat or a curved Riemannian manifold. A particular example with a great relevance in statistical and physical applications is the \emph{exponential family}\footnote{Classical statistical ensembles as canonical and Gran canonical ensembles belong to the exponential family (\ref{Exponential}).}:
\begin{equation}\label{Exponential}
dp(I|\theta)=\exp\left[P(\theta)-\theta^{\alpha}A_{\alpha}(I)+B(I)\right]dI.
\end{equation}
According to \emph{Amari's $\sigma$-connections} \cite{Amari}:
\begin{equation}\label{Amari.Connection}
\Gamma^{(\sigma)}_{\alpha\beta\gamma}(\theta)=\left\langle\left(\frac{\partial\upsilon_{\beta}(\mathcal{I}|\theta)}{\partial\theta^{\alpha}}+
\frac{1-\sigma}{2}\upsilon_{\alpha}(\mathcal{I}|\theta)\upsilon_{\beta}(\mathcal{I}|\theta)\right)\upsilon_{\gamma}(\mathcal{I}|\theta)
\right\rangle.
\end{equation}
the statistical manifold $\mathcal{P}$ associated with the exponential family (\ref{Exponential}) is trivially flat when the connection parameter $\sigma=\pm1$. However, the statistical manifold $\mathcal{P}$ could be a curved manifold for other values of connection parameter $\sigma$. An special case is $\sigma=0$, which corresponds to the Levi-Civita connection (\ref{Levi-Civita}) associated with the inference metric tensor $g_{\alpha\beta}(\theta)\equiv-\partial^{2}P(\theta)/\partial\theta^{\alpha}\partial\theta^{\beta}$. Without mattering about the geometry of the statistical manifold $\mathcal{P}$, the one-dimensional statistical manifold $\mathcal{M}_{\theta}$ is always \emph{diffeomorphic} to the one-dimensional Euclidean manifold $\mathbb{E}$. Clearly, the curvature notion is only admissible for manifolds with dimension $n\geq 2$, and hence, the one-dimensional manifold $\mathcal{M}_{\theta}$ must exhibit a \emph{flat geometry}. As expected, this type of distribution functions represents the simplest application framework of fluctuation geometry.

The invariant volume element (\ref{inv.volume}) of the one-dimensional manifold $\mathcal{M}_{\theta}$ can be rewritten in term of the statistical distance (\ref{fluct.dist}) as follows:
\begin{equation}\label{dm1}
d\mu(I|\theta)=\sqrt{g_{11}(I|\theta)/2\pi}dI\equiv ds/\sqrt{2\pi},
\end{equation}
where $g_{11}(I|\theta)$ denotes the only component of the metric tensor. One can apply the Riemannian gaussian representation (\ref{UGR}) instead of performing the integration of the set of covariant partial differential equations (\ref{cov.equation}). For convenience, let us firstly introduce the coordinate reparametrization $s(I|\theta):\mathcal{R}_{I}\rightarrow \mathcal{R}_{s}$ defined by the distance $\mathfrak{D}_{\theta}(I|\bar{I})\equiv \ell(I)$ as follows:
\begin{equation}\label{repar}
s(I|\theta)=\left\{
                     \begin{array}{cc}
                       -\mathfrak{D}_{\theta}(I|\bar{I}) & \mbox{for }I<\bar{I}, \\
                       \mathfrak{D}_{\theta}(I|\bar{I}) & \mbox{for }I\geq\bar{I}, \\
                     \end{array}
\right.
\end{equation}
where the metric tensor component $g_{11}(s|\theta)\equiv 1$. Using equations (\ref{dm1}) and (\ref{repar}), Riemannian gaussian representation (\ref{UGR}) can be expressed in the coordinate representation $\mathcal{R}_{s}$ as:
\begin{equation}\label{GD.example}
dp(I|\theta)=\rho(I|\theta)dI\equiv\frac{1}{\mathcal{Z}(\theta)}e^{-\frac{1}{2}s^{2}}\frac{ds}{\sqrt{2\pi}}.
\end{equation}
As discussed in \ref{Deriv.FG1D}, the previous expression allows a straightforwardly derivation of the reparametrization function $s(I|\theta)$. Introducing the \emph{cumulant distribution function} $p(I|\theta)$:
\begin{equation}\label{cumulant}
p(I|\theta)=\int^{I}_{I_{min}}dp(I'|\theta)dI',
\end{equation}
the reparametrization function $s(I|\theta)$ is given by:
\begin{equation}\label{map}
s(I|\theta)=\Phi^{-1}\left[p(I|\theta)\right].
\end{equation}
Here, $\Phi^{-1}(z)$ is the inverse of the function $\Phi(z)$:
\begin{equation}\label{Psi}
\Phi(z)=\frac{1}{\sqrt{2\pi}}\int^{z}_{-\infty}e^{-\frac{1}{2}s^{2}}ds\equiv \frac{1}{2}\left(1+\mathrm{erf}(z/\sqrt{2})\right),
\end{equation}
with $\mathrm{erf(z)}$ being the \emph{error function}:
\begin{equation}
\mathrm{erf(z)}=\sqrt{\frac{2}{\pi}}\int^{z}_{0}e^{-x^{2}}dx.
\end{equation}

\begin{figure}
  \includegraphics[width=4.0in]{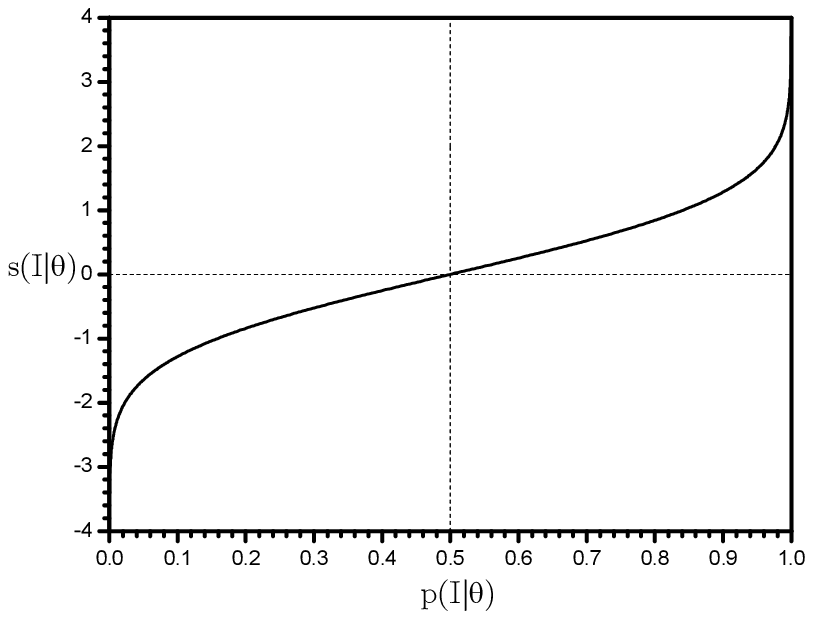}\\
  \caption{Dependence of the reparametrization function $s(I|\theta)$ \emph{versus} the cumulant distribution function $p(I|\theta)$. Notice that the value corresponding to the maximum information potential $s(\bar{I}|\theta)=0$ takes place when the cumulant function $p(\bar{I}|\theta)=1/2$.}\label{PsiFunction.eps}
\end{figure}

The dependence between the reparametrization function $s(I|\theta)$ and the cumulant distribution function $p(I|\theta)$ is illustrated in figure \ref{PsiFunction.eps}. By definition of the reparametrization function (\ref{repar}), the point $\bar{I}$ with maximum information potential $\mathcal{S}(I|\theta)$ corresponds to the condition $s(\bar{I}|\theta)=0$. According to equation (\ref{map}), the information potential $\mathcal{S}(I|\theta)$ exhibits its maximum value at the point $\bar{I}$ where the cumulant distribution function (\ref{cumulant}) reaches the value $p(\bar{I}|\theta)=1/2$. Moreover, the admissible values of the variable $s$ belong to the entire real space $\mathbb{R}$, $-\infty\leq s \leq +\infty$. The normalization of the gaussian distribution (\ref{GD.example}) implies that the gaussian partition function $\mathcal{Z}(\theta)\equiv 1$. Thus, the information potential $\mathcal{S}(I|\theta)$ is given by:
\begin{equation}\label{IP}
\mathcal{S}(I|\theta)\equiv-s^{2}(I|\theta)/2,
\end{equation}
which is a non positive function that diverges at the boundary $\partial \mathcal{M}_{\theta}$ of the statistical manifold $\mathcal{M}_{\theta}$ (at the boundary points $I_{min}$ and $I_{max}$ in the coordinate representation $\mathcal{R}_{I}$.). The only component of the metric tensor $g_{ii}(I|\theta)$ can be expressed as follows:
\begin{equation}\label{metric.example}
g_{11}(I|\theta)=2\pi\rho^{2}(I|\theta)\exp\left[s^{2}(I|\theta)\right].
\end{equation}
\begin{figure}
  \includegraphics[width=4.0in]{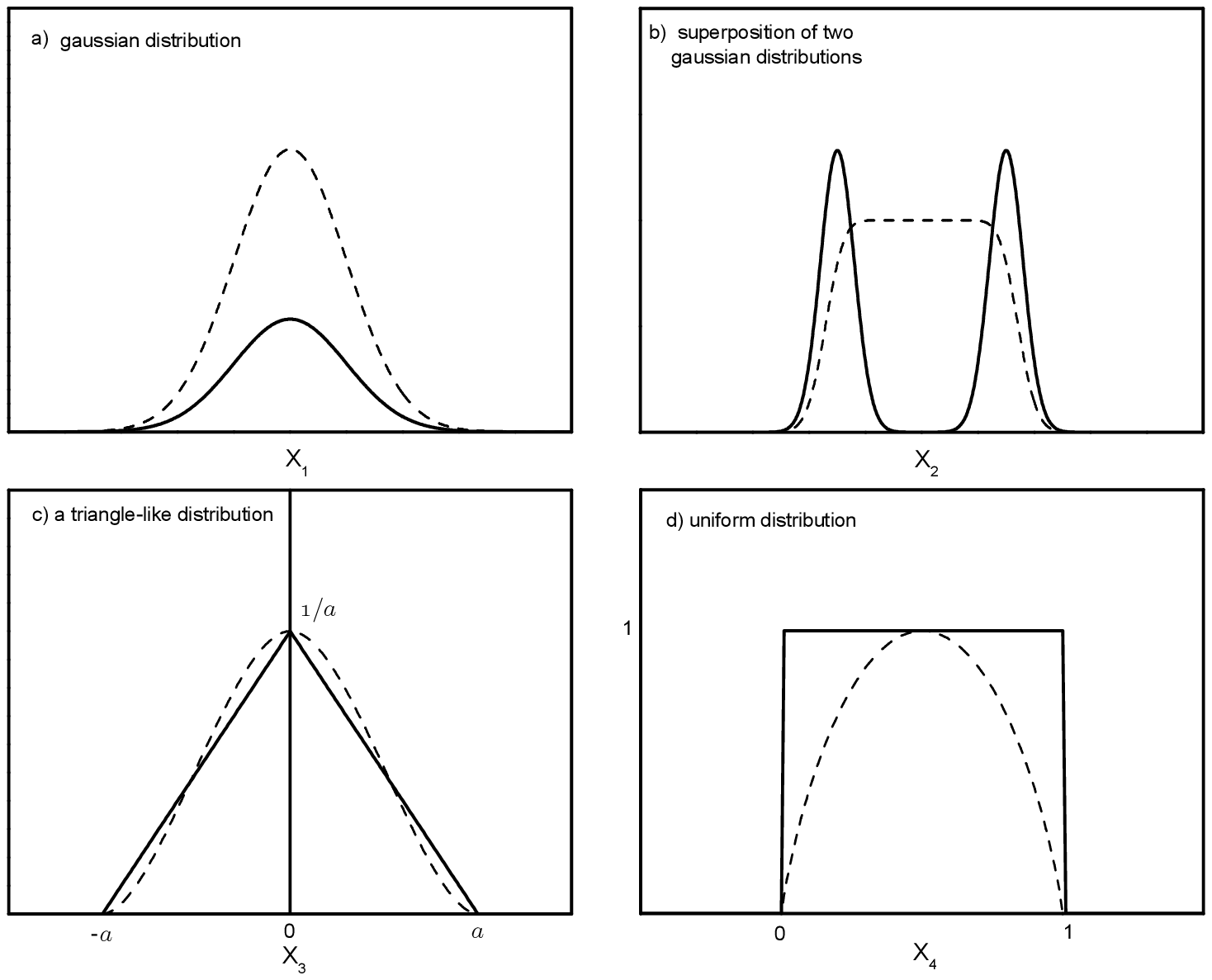}\\
  \caption{Comparison between probability density function $\rho(I|\theta)$ (solid line) and probability weight $\omega(I|\theta)$ (dashed line) for some simple distribution functions: Panel a) A gaussian distribution defined on the one-dimensional real space $\mathbb{R}$. Panel b) The superposition of two different gaussian distributions defined on the one-dimensional real space $\mathbb{R}$. Panel c) a triangle-like distribution defined on the real segment $[-a,a]$. Panel d) A uniform distribution defined on the real segment $[0,1]$. Despite their different appearance, all these distributions are \emph{differmorphic}, that is, they are equivalent from the viewpoint of fluctuation geometry.}\label{equivalence.eps}
\end{figure}
For illustrative purposes, it is shown in figure \ref{equivalence.eps} a comparison between the probability density function $\rho(I|\theta)$ and the probability weight $\omega(I|\theta)=\rho(I|\theta)\sqrt{\left|2\pi g^{11}(I|\theta)\right|}\equiv\exp\left[\mathcal{S}(I|\theta)\right]$ for some simple distribution functions $dp(I|\theta)$. While the probability density $\rho(I|\theta)$ can be a \emph{multimodal function} in certain coordinate representations $\mathcal{R}_{I}$ (as the case shown in panel \ref{equivalence.eps}.b), the probability weight $\omega(I|\theta)$ is always a \emph{monomodal scalar function} as a consequence of \textbf{Theorem 3}.

Summarizing, the present analysis demonstrates that any parametric family $dp(I|\theta)$ defined on an one-dimensional statistical manifold $\mathcal{M}_{\theta}$ can always be mapped onto the gaussian distribution function (\ref{GD.example}) using the reparametrization function (\ref{map}). Consequently, all these distribution functions (as the ones shown in figure \ref{equivalence.eps}) are \emph{diffeomorphic} among them. In other words, all them can be regarded as an \emph{abstract gaussian distribution function} defined on the Euclidean manifold $\mathbb{E}$, but expressed in different coordinate representations $\mathcal{R}_{I}$\footnote{The concept of \emph{diffeomorphic distribution functions} will be considered in subsection \ref{information} to discuss the notion of \emph{intrinsic differential entropy} of a statistical manifold $\mathcal{M}_{\theta}$.}. On the other hand, the relationship between the reparametrization function $s(I|\theta)$ and the cumulant distribution function $p(I|\theta)$ evidences the \emph{purely statistical significance} of the distance notion (\ref{fluct.dist}). As expected, fluctuation geometry establishes a direct correspondence between the geometrical description of the statistical manifold $\mathcal{M}_{\theta}$ and its probabilistic description. For example, the following geometrical and probabilistic inequalities:
\begin{equation}
\mathfrak{D}_{\theta}(I|\bar{I})<\varepsilon\mbox{ and }\left|p(I|\theta)-1/2\right|\leq \frac{1}{2}\mathrm{erf}(\varepsilon/\sqrt{2})
\end{equation}
are fully equivalent.

\subsection{Generalization to the $n$-dimensional Euclidean manifold $\mathbb{E}^{n}$}

Let us suppose that the parametric family $dp(I|\theta)$ can be factorized into independent distribution functions $dp^{(i)}(I^{i}|\theta)$ for each stochastic variable $I^{i}$:
\begin{equation}\label{decomposition}
dp(I|\theta)=\prod_{i}dp^{(i)}(I^{i}|\theta).
\end{equation}
Accordingly, its associated $n$-dimensional statistical manifold $\mathcal{M}_{\theta}$ can be decomposed as the \emph{external product} of the set of one-dimensional statistical manifolds $\left\{\mathcal{E}^{i}_{\theta}\right\}$ as follows:
\begin{equation}
\mathcal{M}_{\theta}=\mathcal{E}^{1}_{\theta}\otimes \mathcal{E}^{2}_{\theta}\ldots \otimes\mathcal{E}^{n}_{\theta},
\end{equation}
and hence, $\mathcal{M}_{\theta}$ is diffeomorphic to the $n$-dimensional Euclidean manifold $\mathbb{E}^{n}$. The results obtained in the previous subsection are straightforwardly extended to the present situation considering that the information potential $\mathcal{S}(I|\theta)$ is additive:
\begin{equation}
\mathcal{S}(I|\theta)=\sum^{n}_{i=1}\mathcal{S}^{(i)}(I^{i}|\theta),
\end{equation}
while the metric tensor $g_{ij}(I|\theta)$ is diagonal:
\begin{equation}
ds^{2}=\sum^{n}_{i=1}g_{ii}(I^{i}|\theta)(dI^{i})^{2}.
\end{equation}
Here, the functions $\mathcal{S}^{(i)}(I^{i}|\theta)$ and $g_{ii}(I^{i}|\theta)$ are obtained from the probability densities $\rho^{(i)}(I^{i}|\theta)$ using equations (\ref{cumulant}), (\ref{map}), (\ref{IP}) and (\ref{metric.example}).

\section{Implications of fluctuation geometry in statistics and physics}\label{Relevance}

\subsection{Comparison between fluctuation geometry and inference geometry}

As naturally expected, the distance notion of inference geometry (\ref{inf.dist}) allows to define a \emph{statistical distance $\mathfrak{D}(\vartheta|\theta)$ between two members} (two different distribution functions) of the parametric family (\ref{DF}), e.g.: considering the arc-length of the geodesics that connects the points $\theta$ and $\vartheta\in\mathcal{P}$. According to asymptotic formula (\ref{asympt.inf}), this statistical distance is associated with the \emph{distinguishing probability} of these distribution functions during a statistical inferential procedure. Conversely, the distance notion of fluctuation geometry (\ref{fluct.dist}) allows to define a \emph{statistical distance $\mathfrak{D}_{\theta}(I_{1}|I_{2})$ between two sets of values of the stochastic variables $I$, which are described by a given member} (a specific distribution function) \emph{of the parametric family} (\ref{DF}). At first glance, the approximation formula (\ref{gaussian.fluct}) can be regarded as a \emph{counterpart expression} of the asymptotic distribution (\ref{asympt.inf}) of inference geometry. Such an analogy clarifies the \emph{relevance} of the distance notion (\ref{fluct.dist}): this second statistical distance is associated with the \emph{occurrence probability} of a small fluctuation $\Delta I=I-\bar{I}$ around the point $\bar{I}$ with maximum information potential $\mathcal{S}(I|\theta)$. Remarkably, the asymptotic formula (\ref{gaussian.fluct}) is the crudest approximation of the Riemannian gaussian representation (\ref{UGR}). According to this rigorous result of fluctuation geometry, the statistical distance $\mathfrak{D}_{\theta}(I|\bar{I})$ is simply a measure of the \emph{relative occurrence probability} in regard to the point with maximum information potential $\bar{I}$:
\begin{equation}
\mathfrak{D}^{2}_{\theta}(I|\bar{I})\equiv-2\log\left[\frac{\omega(I|\theta)}{\omega(\bar{I}|\theta)}\right],
\end{equation}
with $\omega(I|\theta)$ being the probability weight (\ref{scalar.weight}).

Although inference geometry and fluctuation geometry are two counterpart approaches, they provide \emph{different qualitative information} about the statistical properties of a given parametric family (\ref{DF}). On one hand, the first theory provides geometric information concerning to the inference of the control parameters $\theta$ of a given parametric family (\ref{DF}). On the other hand, the second theory provides geometric information about the fluctuating behavior of the stochastic variables $I$ for a given distribution function of the parametric family (\ref{DF}). Noteworthy that the term \emph{geometric information} has a special meaning here: these geometric theories consider those properties of the statistical manifolds $\mathcal{P}$ and $\mathcal{M}_{\theta}$ that are independent on their specific coordinate representations $\mathcal{R}_{\theta}$ and $\mathcal{R}_{I}$. Additionally, these two statistical geometries differ in regard to their \emph{application frameworks}. Inference geometry only demands the \emph{continuous character} of the statistical manifold $\mathcal{P}$ of control parameters $\theta$. Therefore, this type of geometry can be introduced for a \emph{parametric family of distribution functions} $p(X|\theta)$ \emph{defined on a set of discrete variables} $X=\left\{X_{k}\right\}$:
\begin{equation}\label{discrete}
p(X|\theta)=\left\{p(X_{k}|\theta)|k\in\mathbb{Z}\right\}.
\end{equation}
Conversely, fluctuation geometry only demands the \emph{continuous character} of the manifold $\mathcal{M}_{\theta}$ of stochastic variables $I$. Therefore, this geometry can be introduced for a \emph{continuous distribution function without control parameters}:
\begin{equation}
dp(I)=\rho(I)dI.
\end{equation}
As expected, the simultaneous definition of inference geometry and fluctuation geometry is only possible for parametric families (\ref{DF}) defined on continuous statistical manifolds $\mathcal{M}_{\theta}$ and $\mathcal{P}$.

A simple look to equations (\ref{Fisher}) and (\ref{cov.equation}) allows us to realize that these geometric theories have a different \emph{amenable character}. In particular, the metric tensor $g_{\alpha\beta}(\theta)$ of inference geometry (\ref{Fisher}) is very easy to obtain, either from the analytical or numerical calculation of these integrals. Conversely, the metric tensor $g_{ij}(I|\theta)$ of fluctuation geometry should be obtained solving a set of covariant partial differential equations (\ref{cov.equation}), whose admissible solutions must obey certain boundary conditions. Actually, this latter mathematical procedure can be a hard task for a manifold $\mathcal{M}_{\theta}$ with a nontrivial geometry. However, once obtained the metric tensors $g_{\alpha\beta}(\theta)$ and $g_{ij}(I|\theta)$, the amenable character of these geometric theories changes in a radical way: \emph{fluctuation geometry turns much more amenable than inference geometry}. For example, a modest mathematical effort has been devoted to arrive at the rigorous \textbf{Theorems 2-4} and their associated implications as the Riemannian gaussian representation (\ref{UGR}). Conversely, inference geometry has to deal with \emph{a very serious difficulty}: there not exist a general way to relate a given parametric family (\ref{DF}) and the \emph{unbiased estimators} $\hat{\theta}$ of its control parameters $\theta$. Even the calculation of the distribution function (\ref{asymp.dist}) of the \emph{efficient unbiased estimators} $\hat{\theta}_{eff}$ is not a easy task. In particular, the asymptotic distribution (\ref{asympt.inf}) is a direct application of the \emph{central limit theorem}, that is, an approximation formula for a statistical inference with a large but finite number $m$ of outcomes $\mathcal{I}=\left\{I^{(1)},I^{(2)},...I^{(m)}\right\} $. A very important question during the historical development of inference geometry is the so-called \emph{higher-order asymptotic theory of statistical estimation} \cite{Amari}, where a fundamental task is the improvement of approximation formula (\ref{asympt.inf}) considering a $1/m$-power expansion based on the intrinsic geometry of the manifold $\mathcal{P}$. A counterpart of Riemannian gaussian representation (\ref{UGR}) in the framework of inference geometry is unknown in the literature, at least, from the knowledge of the present author.

\subsection{On the notion of information entropy for a continuous distribution}\label{information}

As discussed elsewhere \cite{Thomas}, the \emph{information entropy} $S\left[p|\theta\right]$ associated with the discrete distribution function (\ref{discrete}) is written as follows:
\begin{equation}\label{statistical.entropy}
S\left[p|\theta\right]=-\sum_{k}p(X_{k}|\theta)\log p(X_{k}|\theta).
\end{equation}
Conceptually, information entropy is considered as a measure of the \emph{unpredictability} associated with a random variable $X$. Interestingly, its counterpart extension for a continuous distribution function:
\begin{equation}\label{arb.DF}
dQ(I)=q(I)dI,
\end{equation}
the so-called \emph{differential entropy}:
\begin{equation}\label{differential.entropy}
S^{(\mathcal{R}_{I})}_{de}\left[Q|\mathcal{M}_{\theta}\right]=-\int_{\mathcal{M}_{\theta}}q(I)\log q(I) dI,
\end{equation}
undergoes an important geometric inconsistence: its definition crucially depends on the coordinate representation $\mathcal{R}_{I}$ of the statistical manifold $\mathcal{M}_{\theta}$:
\begin{equation}
S^{(\mathcal{R}_{I})}_{de}\left[Q|\mathcal{M}_{\theta}\right]\neq S^{(\mathcal{R}_{\Theta})}_{de}\left[Q|\mathcal{M}_{\theta}\right]=-\int_{\mathcal{M}_{\theta}}q(\Theta)\log q(\Theta) d\Theta.
\end{equation}
In general, the expectation values of scalar functions defined on the statistical manifold $\mathcal{M}_{\theta}$ are only independent on the coordinate representations:
\begin{equation}
A(I)=A(\Theta)\Rightarrow\left\langle A(I)\right\rangle=\left\langle A(\Theta)\right\rangle.
\end{equation}
However, the probability density function $q(I)$ is simply a \emph{tensorial density}, whose values and general mathematical behavior crucially depend on the concrete coordinate representation $\mathcal{R}_{I}$ of the statistical manifold $\mathcal{M}_{\theta}$. Consequently, the consideration of the quantity $\mathcal{I}_{c}(I)=-\log q(I)$ as a local measure of the \emph{information content} is ill-defined from the geometric viewpoint because of it violates the requirement of covariance under the coordinate reparametrization $\Theta(I):\mathcal{R}_{I}\rightarrow\mathcal{R}_{\Theta}$ of the statistical manifold  $\mathcal{M}_{\theta}$. Despite their apparent similarity, the differential geometry (\ref{differential.entropy}) is not a good generalization of the statistical entropy (\ref{statistical.entropy}) for the framework of continuous distribution functions. For example, the differential entropy (\ref{differential.entropy}) does not obey other properties of its discrete counterpart (\ref{statistical.entropy}), in particular, the positive definition $S\left[p|\theta\right]\geq 0$.

An attempt to overcome some of the above inconsistences was developed by Jaynes \cite{Jaynes}. According to this author, the correct formula for the information entropy of a continuous distribution function can be derived taking the \emph{limit of increasingly dense discrete distributions}. Specifically, Jaynes proposed to start from of a set of $n$ discrete points $\mathcal{S}_{n}=\left\{I_{i}\right\}\subset\mathcal{M}_{\theta}$, which density $\gamma_{n}(I)=n^{-1}\sum^{n}_{i=1}\delta \left(I-I_{i}\right)$ approaches a certain function $\gamma(I)$ in the limit $n\rightarrow\infty$:
\begin{equation}
\gamma(I)=\lim_{n\rightarrow\infty}\gamma_{n}(I).
\end{equation}
The density function $\gamma(I)$ is referred to as the \emph{invariant measure}. Combining the previous argument and the discrete definition of information entropy (\ref{statistical.entropy}), this author arrived at the following correction for the differential entropy:
\begin{equation}\label{Jaynes.entropy}
\mathcal{S}^{\gamma}_{de}\left[Q|\mathcal{M}_{\theta}\right]=-\int_{\mathcal{\mathcal{M}_{\theta}}}q(I)\log\left[\frac{q(I)}{\gamma(I)}\right]dI.
\end{equation}
At first glance, the present formula is similar to but conceptually distinct from the (negative of the) \emph{Kullback-Leibler divergence} \cite{Pardo}:
\begin{equation}\label{KL}
D_{KL}(Q|P)=\int_{\mathcal{\mathcal{M}_{\theta}}}q(I)\log\left[\frac{q(I)}{p(I)}\right]dI.
\end{equation}
As many other divergences considered in statistics \cite{Pardo}, Kullback-Leibler divergence (\ref{KL}) is a measure of the \emph{separation} of a distribution function $dQ(I)=q(I)dI$ to a reference distribution $dP(I)=p(I)dI$. In the formula (\ref{Jaynes.entropy}), however, the invariant measure $\gamma(I)$ need not to be a \emph{probability density}, but simply a density. In particular, it need not satisfied the normalization condition:
\begin{equation}
\int_{\mathcal{M}_{\theta}}\gamma(I)dI\neq 1.
\end{equation}

Although Jaynes' differential entropy (\ref{Jaynes.entropy}) is invariant under coordinate reparametrizations, the success achieved with this correction formula is only partial. In fact, Jaynes was unable to provide a general criterium to precise the invariant measure $\gamma(I)$ for a concrete application. Referring to this ambiguity, he recognized that \emph{... the following arguments can be made as rigorous as we please, but at considerable sacrifice of clarity} \cite{Jaynes}. Remarkably, the pre-existence of a \emph{Riemannian structure} defined on the statistical manifold $\mathcal{M}_{\theta}$ introduces a natural choice for the invariant measure $\gamma(I)$. While the probability density $q(I)$ of a distribution function $dQ(I)$ depends on the coordinate representation $\mathcal{R}_{I}$, the notion of \emph{probability weight}\footnote{The notion of probability weight was considered in equation (\ref{scalar.weight}) to introduce the information potential $\mathcal{S}(I|\theta)$.}:
\begin{equation}\label{proper.density}
q_{g}(I)=q(I)\sqrt{\left|2\pi g^{ij}(I|\theta)\right|}
\end{equation}
represents a scalar function defined on the statistical manifold $\mathcal{M}_{\theta}$. Using the probability weight $q_{g}(I)$ instead of the probability density $q(I)$, the quantity $\mathfrak{I}_{c}(I|\theta)=-\log q_{g}(I)$ can be introduced as \emph{a local invariant measure of the information content}. Thus, a more appropriate generalization of information entropy for a continuous distribution function is given by:
\begin{equation}\label{Cov.Ent}
\mathcal{S}^{g}_{de}\left[Q|\mathcal{M}_{\theta}\right]=\left\langle\mathfrak{I}_{c}(I|\theta)\right\rangle=-\int_{\mathcal{M}_{\theta}}q_{g}(I)\log q_{g}(I)d\mu(I|\theta),
\end{equation}
where the index $g$ denotes the Riemannian structure of the statistical manifold $\mathcal{M}_{\theta}$. Noteworthy that equation (\ref{Cov.Ent}) is a particular case of Jaynes' differential entropy (\ref{Jaynes.entropy}), where the invariant measure $\gamma(I)$ is determined by the metric tensor $g_{ij}(I|\theta)$ of the statistical manifold $\mathcal{M}_{\theta}$:
\begin{equation}\label{Gamma.geo}
\gamma(I)=\sqrt{\left|g_{ij}(I|\theta)/2\pi\right|}\equiv 1/\sqrt{\left|2\pi g^{ij}(I|\theta)\right|}.
\end{equation}
According to Jaynes' argument, the invariant measure (\ref{Gamma.geo}) can be obtained as the limit of increasingly dense subset of points $\mathcal{S}_{n}$ that are \emph{uniformly distributed} on the statistical manifold $\mathcal{M}_{\theta}$, that is, a distribution function whose probability weight $\gamma_{g}(I|\theta)\equiv 1$.

The geometric differential entropy (\ref{Cov.Ent}) depends both on the distribution function $dQ(I)$ as well as the Riemannian structure of the statistical manifold $\mathcal{M}_{\theta}$. According to postulates of fluctuation geometry, the Riemannian structure of the statistical manifold $\mathcal{M}_{\theta}$ is associated with a \emph{reference distribution function} $dp(I|\theta)$, specifically, the distribution function (\ref{EinsteinPostulate}) derived from the knowledge of the information potential $\mathcal{S}(I|\theta)$. Therefore, it is worth to distinguish between two different notions of differential entropy:
\begin{itemize}
  \item The differential entropy $\mathcal{S}^{g}_{de}\left[Q|\mathcal{M}_{\theta}\right]$ of an arbitrary distribution function $dQ(I)$ defined on a statistical manifold $\mathcal{M}_{\theta}$, which is endowed of a \emph{pre-existent} Riemannian structure.
  \item The notion of \emph{intrinsic differential entropy} $\mathcal{S}^{g}_{de}\left[\mathcal{M}_{\theta}\right]$ of a statistical manifold $\mathcal{M}_{\theta}$, that is, the differential entropy $\mathcal{S}^{g}_{de}\left[p|\mathcal{M}_{\theta}\right]$ of the distribution function $dp(I|\theta)$ associated with the Riemannian structure of the statistical manifold $\mathcal{M}_{\theta}$.
\end{itemize}

The intrinsic differential entropy $\mathcal{S}^{g}_{de}\left[\mathcal{M}_{\theta}\right]$ of the statistical manifold $\mathcal{M}_{\theta}$ is given by the negative of the expectation value of the information potential $\mathcal{S}(I|\theta)$:
\begin{equation}
\mathcal{S}^{g}_{de}\left[\mathcal{M}_{\theta}\right]\equiv -\left\langle\mathcal{S}(I|\theta)\right\rangle,
\end{equation}
which can be rewritten using equation(\ref{E1}) as follows:
\begin{equation}
\mathcal{S}^{g}_{de}\left[\mathcal{M}_{\theta}\right]= \frac{1}{2}\left\langle \mathfrak{D}^{2}_{\theta}(I|\bar{I})\right\rangle-\mathcal{P}(\theta).
\end{equation}
Accordingly, $\mathcal{S}^{g}_{de}\left[\mathcal{M}_{\theta}\right]$ is a \emph{global geometric measure} of an statistical manifold $\mathcal{M}_{\theta}$, which depends on its topological properties and Riemannian structure, as well as the position of the point $\bar{I}$ with maximum information potential. In particular, if the statistical manifold $\mathcal{M}_{\theta}$ can be decomposed into two independent Riemannian manifolds $\mathcal{A}$ and $\mathcal{B}$ as $\mathcal{M}_{\theta}=\mathcal{A}\otimes\mathcal{B}$, its intrinsic differential entropy $\mathcal{S}^{g}_{de}\left[\mathcal{M}_{\theta}\right]$ is additive:
\begin{equation}\label{additivity}
\mathcal{S}^{g}_{de}\left[\mathcal{M}_{\theta}\right]=\mathcal{S}^{g_{\mathcal{A}}}_{de}\left[\mathcal{A}\right]
+\mathcal{S}^{g_{\mathcal{B}}}_{de}\left[\mathcal{B}\right],
\end{equation}
where $g_{\mathcal{A}}$ and $g_{\mathcal{B}}$ denote their respective Riemannian structures.

Before we end this section, it is worth remarking that the requirement of covariance is a strong constraint. The existence of this symmetry in the differential entropy (\ref{Cov.Ent}) implies that \emph{diffeomorphic distribution functions exhibit the same value of their intrinsic differential entropies}. As already commented, all distribution functions illustrated in figure \ref{equivalence.eps} are diffeomorphic. In fact, their statistical manifolds $\mathcal{M}_{\theta}$ are diffeomorphic to the one-dimensional Euclidean manifold $\mathbb{E}$. Thus, their respective intrinsic differential entropies exhibit the same value $\mathcal{S}^{g}_{de}\left[\mathbb{E}\right]=1/2$. This result is easy to obtain using the gaussian distribution (\ref{GD.example}) associated with the coordinate representation $\mathcal{R}_{s}$ of the statistical manifold $\mathcal{M}_{\theta}$. Moreover, the $n$-dimensional manifold $\mathcal{M}_{\theta}$ associated with a distribution function obeying the decomposition (\ref{decomposition}) has an intrinsic information entropy $\mathcal{S}^{g}_{de}\left[\mathbb{E}^{n}\right]=n/2$, which is a direct consequence of the property (\ref{additivity})

\subsection{Riemannian reformulation of Einstein fluctuation theory}

Classical fluctuation theory starts from \emph{Einstein postulate} \cite{Reichl}:
\begin{equation}\label{EP}
dp(I|\theta)=\mathcal{A}e^{S(I|\theta)}dI,
\end{equation}
which describes the fluctuating behavior of a set of macroscopic observables $I$ in an equilibrium situation driven by certain control parameters $\theta$\footnote{The Boltzmann's constant $k$ is set as the unity, $k=1$.}. Here, $S(I|\theta)$ denotes the entropy of a closed system, while $\mathcal{A}$ is a normalization constant. Since the function $\rho(I|\theta)=\mathcal{A}e^{S(I|\theta)}$ is a tensorial density, the entropy $S(I|\theta)$ considered in Einstein postulate (\ref{EP}) behaves under a coordinate reparametrization $\Theta(I):\mathcal{R}_{I}\rightarrow\mathcal{R}_{\Theta}$ as follows:
\begin{equation}
S(\Theta|\theta)=S(I|\theta)-\log\left|\frac{\partial\Theta}{\partial I}\right|,
\end{equation}
and hence, it is not a scalar function. This feature is a direct contradiction with the thermodynamic relevance of entropy as a \emph{state function}\footnote{A state function is a property of a system that depends only on its current state, not on the way in which the system acquired that state. Geometrically, the value of a state function should not depend on the coordinate representation employed to describe that state, that is, it should be a \emph{scalar function}.}. Consequently, Einstein postulate (\ref{EP}) is \emph{ill-defined from the geometric viewpoint}.

As already discussed in a recent paper \cite{Vel.GEO}, the previous inconsistence disappears when one redefines Einstein postulate into the covariant form (\ref{EinsteinPostulate}). Thus, the entropy of a closed system should be identified with the information potential $\mathcal{S}(I|\theta)$ of fluctuation geometry up to the precision of an additive constant. As expected, the covariant distribution function (\ref{EinsteinPostulate}) also depends on the Riemannian structure of the statistical manifold $\mathcal{M}_{\theta}$ of the macroscopic observables $I$. Such an ambiguity disappears considering the relationship between the metric tensor $g_{ij}(I|\theta)$ and the negative of the \emph{covariant Hessian} of the scalar entropy $\mathcal{S}(I|\theta)$, equation (\ref{cov.EH}). Consequently, the system fluctuating behavior is fully determined by the knowledge of the entropy $\mathcal{S}(I|\theta)$. From the physical viewpoint, the constraint (\ref{cov.EH}) between the metric tensor $g_{ij}(I|\theta)$ and the entropy $\mathcal{S}(I|\theta)$ is very relevant. In fact, such a choice ensures the \emph{geodesic character} of the \emph{system hydrodynamic equations} \cite{Vel.GEO}:
\begin{equation}\label{hydro}
\frac{dI^{i}(s)}{ds}=-\upsilon^{i}(I|\theta).
\end{equation}
Here, the parameter $s$ denotes the arc-length of the curve of hydrodynamic relaxation $I(s)\in\mathcal{M}_{\theta}$. This curve can be characterized by the tangent vector field $\xi(s)$ with contravariant components:
\begin{equation}
\xi^{i}(s)=\frac{dI^{i}(s)}{ds},
\end{equation}
which is unitary vector field, $g_{ij}(I|\theta)\xi^{i}(s)\xi^{j}(s)\equiv 1$. The tangent vector field $\xi(s)$ is oriented in the same direction of the \emph{generalized restituting force} $\zeta(I|\theta)$ with covariant components $\zeta_{i}(I|\theta)=\partial\mathcal{S}(I|\theta)/\partial I^{i}\equiv-\psi_{i}(I|\theta)$. As expected, the generalized restituting force is directed towards the equilibrium configuration $\bar{I}$ (the point with maximum entropy), whose direction is determined by the negative of the unitary vector field $\upsilon^{i}(I|\theta)$ introduced in equation (\ref{unitary}). As already shown in the proof of \textbf{Theorem 4}, the unitary vector field $\upsilon^{i}(I|\theta)$ obeys geodesic differential equations (\ref{geodesic}). In physics, the geodesics are regarded as the \emph{natural motions} in theories with a Riemannian formulation, as example, in general relativity theory. This nontrivial result establishes an interesting connection of the present Riemannian reformulation of equilibrium fluctuation theory and its possible nonequilibrium generalization.

\subsection{Relation with Ruppeiner's geometry of thermodynamics}

Ruppeiner proposed in the past a Riemannian geometry for thermodynamics \cite{Ruppeiner}. Despite its large history in the literature, Ruppeiner's geometry undergoes some inconsistencies. The most important is that this approach assumes the scalar character of entropy $S(I|\theta)$ employed in Einstein postulate (\ref{EP}). Ruppeiner recognized himself the restricted applicability of this consideration (see in subsection II.B in his paper in Review of Modern Physics \cite{Ruppeiner}). However, he justified its application as a \emph{good approximation} in the framework of large thermodynamic systems, whose fluctuating behavior can be described within the \emph{gaussian approximation}:
\begin{equation}\label{Rup.Gaussian}
dp(I|\bar{I})\simeq \exp\left[-\frac{1}{2}g_{ij}(\bar{I})\Delta I^{i}\Delta I^{j}\right]\sqrt{\left|\frac{g_{ij}(\bar{I})}{2\pi}\right|}dI.
\end{equation}
Here, $\Delta I^{i}=I^{i}-\bar{I}^{i}$ denotes a small deviation from the most likely (equilibrium) state $\bar{I}$, while $g_{ij}(\bar{I})$ represents the \emph{thermodynamic metric tensor} \cite{Ruppeiner}:
\begin{equation}\label{thermo.tensor}
g_{ij}(\bar{I})=-\frac{\partial^{2}S(\bar{I}|\theta)}{\partial I^{i}\partial I^{j}}.
\end{equation}
Notice that the thermodynamic metric tensor is expressed here in terms of the most likely state $\bar{I}$ instead of the control parameters $\theta$. Apparently, the goal of this consideration is to justify the relevance of the distance notion:
\begin{equation}
ds^{2}_{R}=g_{ij}(\bar{I})d\bar{I}^{i}d\bar{I}^{j}
\end{equation}
as a \emph{thermodynamic distance between equilibrium states}. The previous consideration, however, has a restricted applicability even in the framework of large thermodynamic systems. Since the function $\rho(I|\theta)=\mathcal{A}e^{S(I|\theta)/k}$ is a tensorial density, the probability distribution (\ref{EP}) can exhibit two or more maxima $\bar{I}$ for certain values of control parameters $\theta$. Consequently, it is not possible to guarantee the bijective correspondence between the control parameters $\theta$ and the equilibrium states $\bar{I}$ (the most likely values of macroscopic observables). This type of situations is associated with the \emph{phenomenon of ensemble inequivalence} in statistical mechanics, which is observed during the occurrence of \emph{discontinuous phase transitions} \cite{Reichl}.

Fluctuation geometry successfully overcome the previous limitations of Ruppeiner's geometry \cite{Vel.GEO}. As already commented, the key considerations are (i) to assume the covariant redefinition of Einstein postulate (\ref{EinsteinPostulate}) to ensure the scalar character of the entropy $\mathcal{S}(I|\theta)$; and (ii) to generalize the thermodynamic metric tensor (\ref{thermo.tensor}) considering the scalar entropy $\mathcal{S}(I|\theta)$ and the covariant differentiation $D_{i}$:
\begin{equation}
g_{ij}(\bar{I})=-\frac{\partial^{2}S(\bar{I}|\theta)}{\partial I^{i}\partial I^{j}}\rightarrow g_{ij}(I|\theta)=-D_{i}D_{j}\mathcal{S}(I|\theta).
\end{equation}
From this viewpoint, Riemannian reformulation of Einstein fluctuation theory arises as a formal improvement of Ruppeiner's geometry. Rewriting the previous relation as follows:
\begin{equation}\label{dsg}
g_{ij}(I|\theta)=-\frac{\partial^{2}\mathcal{S}(I|\theta)}{\partial I^{i}\partial I^{j}}+\Gamma^{k}_{ij}(I|\theta)\frac{\partial\mathcal{S}(I|\theta)}{\partial I^{k}},
\end{equation}
it is easy to check that the metric tensor $g_{ij}(I|\theta)$ looks like the thermodynamic metric tensor (\ref{thermo.tensor}) at the point $\bar{I}$ with maximum entropy, where $\partial \mathcal{S}(\bar{I}|\theta)/\partial I^{i}=0$. However, the existence and uniqueness of the point $\bar{I}$ is now guaranteed by \textbf{Theorem 3}. Moreover, the metric tensor $g_{ij}(I|\theta)$ of fluctuation geometry is well-defined for any macroscopic state $I\in\mathcal{M}_{\theta}$. While gaussian approximation (\ref{Rup.Gaussian}) is only applicable to large thermodynamic systems with a small fluctuating behavior, Riemannian gaussian representation (\ref{UGR}) is a rigorous result. Noteworthy that the application of fluctuation geometry in classical fluctuation theory cannot be regarded as a Riemannian approach of thermodynamics, but a \emph{Riemannian approach of classical statistical mechanics}\footnote{Thermodynamics is a macroscopic physical theory that disregards the incidence of fluctuations.}.

\section{Final remarks and open problems}\label{Final}
Fluctuation geometry was proposed in this work as a counterpart approach of inference geometry. This new geometry allows the introduction of the distance notion $ds^{2}=g_{ij}(I|\theta)dI^{i}dI^{j}$ to characterize a \emph{statistical distance} between two different values of the stochastic variables $I$, whose behavior is described by a member of a parametric family of continuous distribution functions $dp(I|\theta)=\rho(I|\theta)dI$. The metric tensor $g_{ij}(I|\theta)$ has been derived starting from a set of axioms, which lead to a set of covariant differential equations (\ref{cov.equation}) written in terms of the probability density $\rho(I|\theta)$. The main consequence is the possibility to rephrase the probability description in terms of purely geometric notions, as the case of Riemannian gaussian distribution (\ref{UGR}). Thus, the statistical description can be equivalently performed using the language of Riemannian geometry, and hence, fluctuation geometry represents an alternative framework for applying the powerful tools of differential geometry for the statistical analysis. As already evidenced, the present approach leads to a reconsideration of the notion of information entropy for a continuous distribution as well as the Riemannian reformulation of Einstein fluctuation theory.

Before we end this section, let us comment some open problems that deserve a special attention in future works. Firstly, it is important to clarify the \emph{existence and uniqueness} of the solution of the covariant differential equations (\ref{cov.equation}). As already evidenced, postulates of fluctuation geometry allow a univocal determination of the Riemannian structure of the statistical manifold $\mathcal{M}_{\theta}$ for the application examples discussed in this work, which exhibit a trivial Euclidean (flat) geometry. Thus, it is necessary to check if such an existence and uniqueness are preserved in a statistical manifold $\mathcal{M}_{\theta}$ with a more complex Riemannian structure.

Secondly, one expects that the \emph{curvature notion} of manifold $\mathcal{M}_{\theta}$ should play a fundamental role from the statistical viewpoint. Some basic arguments suggest that curvature should be associated with the notion of \emph{statistical correlations}. Both curvature notion and the statistical correlations, as example, can only be defined when the dimension $n$ of the statistical manifold $\mathcal{M}_{\theta}$ is equal or larger than two. Interestingly, the existence of a decomposition (\ref{decomposition}) in a parametric family implies the flat character of the statistical manifold $\mathcal{M}_{\theta}$. This possible relevance of the curvature notion of the statistical manifold $\mathcal{M}_{\theta}$ is consistent with some physical analogies. General Relativity theory, as example, identifies gravitational interaction with the curvature of the space-time $\mathbb{M}^{4}$. In the framework of statistical theories as quantum mechanics, the statistical correlations can be regarded as the counterpart of interactions. The gas of non-interacting particles obeying Fermi-Dirac statistics, in particular, manifests \emph{effective repulsion forces} as consequence of the \emph{inter-particle correlations} associated with Pauli's exclusion principle. By analogy, a non-vanishing curvature of a statistical manifold as $\mathcal{M}_{\theta}$ would be associated with the existence of \emph{irreducible statistical correlations}. The analysis of this conjecture will be the main interest of a forthcoming paper.

A third question is to analyze how deep is the analogy between inference geometry and fluctuation geometry. Specifically, it is natural to wonder if each result obtained in one of these theories has a counterpart relation in the other theory. Gaussian approximation (\ref{gaussian.fluct}), as example, can be regarded as a counterpart result of the asymptotic distribution (\ref{asympt.inf}) of inference theory. Starting from the fact that gaussian distribution (\ref{gaussian.fluct}) admits the exact improvement (\ref{UGR}), the underlying analogy strongly suggests the following improvement:
\begin{equation}
dQ^{m}(\vartheta|\theta)=\frac{1}{\mathcal{Z}(\theta)}\exp\left[-\frac{1}{2}
\ell^{2}(\vartheta)\right]\sqrt{\frac{g_{\alpha\beta}(\vartheta)}
{2\pi}}d\vartheta
\end{equation}
for the asymptotic distribution (\ref{asympt.inf}) of inference theory. Here, $\ell(\vartheta)$ should represent the distance $\mathfrak{D}(\vartheta,\theta)$ between the points $\vartheta$ and $\theta$ calculated with the metric tensor $g_{\alpha\beta}(\theta)$ defined on the statistical manifold $\mathcal{P}$ of control parameters $\theta$. It would be interesting to analyze this conjecture. Finally, it would be interesting to analyze other implications of fluctuation geometry in the framework of information theory \cite{Thomas}. A question with a special interest is the application of the notion of intrinsic differential entropy to the problem of \emph{maximum information distributions} \cite{Uffink}.

\appendix
\section{Derivation of the reparametrization function $s(I|\theta)$}\label{Deriv.FG1D}

Let us denote by $s_{1}$ and $s_{2}$ the coordinates that correspond to the boundary points $I_{min}$ and $I_{max}$ in the new coordinate representation $\mathcal{R}_{s}$.  The integration of equation (\ref{GD.example}) yields the following relation:
\begin{equation}
p(I|\theta)=\frac{\Phi(s)-\Phi(s_{1})}{\Phi(s_{2})-\Phi(s_{1})}.
\end{equation}
Here, $p(I|\theta)$ represents the cumulant distribution function (\ref{cumulant}) and $\Phi(z)$ the function (\ref{Psi}). Moreover, it was taken into account that the normalization condition of the distribution function $dp(I|\theta)$ implies the relation $\mathcal{Z}(\theta)\equiv\Phi(s_{2})-\Phi(s_{1})$. Introducing the inverse function $\Phi^{-1}(z)$, the reparametrization function $s(I|\theta):\mathcal{R}_{I}\rightarrow \mathcal{R}_{s}$ can be expressed as follows:
\begin{equation}\label{rep.f1}
s(I|\theta)=\Phi^{-1}\left[\phi+\sigma p(I|\theta)\right],
\end{equation}
where $\phi=\Phi(s_{1})$ and $\sigma=\Phi(s_{2})-\Phi(s_{1})$. Considering the expression (\ref{E1}), the information potential $\mathcal{S}(I|\theta)$ can be expressed as follows:
\begin{equation}\label{IP.der}
\mathcal{S}(I|\theta)=-\log \sigma - \frac{1}{2}s^{2}(I|\theta).
\end{equation}
Both the reparametrization function (\ref{rep.f1}) and the information potential (\ref{IP.der}) depend on the nonnegative constants $\phi$ and $\sigma$, whose values are determined by the boundary points $s_{1}$ and $s_{2}$ in the coordinate representation $\mathcal{R}_{s}$. However, a careful analysis reveals that these parameters cannot admit arbitrary values. Firstly, one should notice that the information potential (\ref{IP.der}) is everywhere finite when $\left|s_{1,2}\right|<\infty$. Taking into account the relationship between the information potential $\mathcal{S}(I|\theta)$ and the probability density $\rho(I|\theta)$:
\begin{equation}
\rho(I|\theta)\equiv\exp\left[\mathcal{S}(I|\theta)\right]\sqrt{\left|\frac{g_{11}(I|\theta)}{2\pi}\right|},
\end{equation}
the vanishing of the probability density $\rho(I|\theta)$ at the boundary $\partial \mathcal{M}_{\theta}$ of the statistical manifold $\mathcal{M}_{\theta}$ (\textbf{Axiom 4}) implies the vanishing of the metric tensor $g_{11}(I|\theta)$ on the boundary $\partial \mathcal{M}_{\theta}$. However, the metric tensor $g_{11}(I|\theta)$ should be non-vanishing everywhere, even, on the boundary $\partial \mathcal{M}_{\theta}$ of the statistical manifold $\mathcal{M}_{\theta}$. The existence of the contravariant metric tensor $g^{ij}(I|\theta)$, in particular, demands the non-vanishing of the metric tensor determinant $\left|g_{ij}(I|\theta)\right|$. For an one dimensional manifold $\mathcal{M}_{\theta}$, this last requirement implies $0<\left|g_{11}(I|\theta)\right|<\infty$. The only way to fulfil such a requirement is to impose the constraints $\phi=0$ and $\sigma=1\Leftrightarrow s_{1}=-\infty$ and $s_{2}= + \infty$, which leads to equation (\ref{map}).

\section*{Acknowledgements}
Velazquez thanks the financial support of CONICYT/Programa Bicentenario de Ciencia y Tecnolog\'{\i}a PSD
\textbf{65} (Chilean agency).

\end{document}